\documentclass{amsart}
\usepackage{amssymb,amscd}  

\newtheorem{theorem}{Theorem}[section]
\newtheorem{lemma}[theorem]{Lemma}
\newtheorem{corollary}[theorem]{Corollary}

\newtheorem{proposition}[theorem]{Proposition}

\theoremstyle{definition}
\newtheorem{definition}[theorem]{Definition}
\newtheorem{example}[theorem]{Example}
\newtheorem{remark}[theorem]{Remark} 
\numberwithin{equation}{section}

\newcommand\qqq{\textbf{[??]}} 

\newcommand\tensor{\otimes}
\newcommand\isomorphic{\cong}

\newcommand\directsum{\oplus}

\DeclareMathOperator{\linalgspan}{span}

\DeclareMathOperator{\Divisor}{div}
\newcommand\union{\cup}
\newcommand\Union{\bigcup}

\newcommand\intersect{\cap}

\newcommand\abs[1]{{\left|#1\right|}}

\newcommand\Projective{{\bf P}} 

\newcommand\fieldK{\mathbb{K}} 
\newcommand\Pinf{{P_\infty}} 
\newcommand\Raff{\mathcal{R}} 
\newcommand\vinf{{v_\Pinf}} 
\DeclareMathOperator{\Pic}{Pic}
\newcommand\LL{\mathcal{L}}
\newcommand\OC{\mathcal{O}_C}

\begin{document}

\title{On Jacobian group arithmetic for typical divisors on curves}

\author{Kamal Khuri-Makdisi}
\address{Mathematics Department, 
American University of Beirut, Bliss Street, Beirut, Lebanon}
\email{kmakdisi@aub.edu.lb}
\thanks{November 14, 2017}

\begin{abstract}
In a previous joint article with F. Abu Salem, we gave efficient
algorithms for Jacobian group arithmetic of ``typical'' divisor
classes on $C_{3,4}$ curves, improving on similar results by other
authors.  At that time, we could only state that a general divisor was
typical, and hence unlikely to be encountered if one implemented these
algorithms over a very large finite field.  This article pins down an
explicit characterization of these typical divisors, for an arbitrary
smooth projective curve of genus $g \geq 1$ having at least one
rational point.  We give general algorithms for Jacobian group
arithmetic with these typical divisors, and prove not only that the
algorithms are correct if various divisors are typical, but also that
the success of our algorithms provides a guarantee that the
resulting output is correct and that the resulting input and/or output
divisors are also typical.  These results apply in particular to our
earlier algorithms 
for $C_{3,4}$ curves.  As a byproduct, we obtain a further speedup of
approximately 15\% on our previous algorithms for $C_{3,4}$ curves.
\end{abstract}

\maketitle

\section{Introduction}
\label{section1}

Let $C$ be a smooth projective algebraic curve of genus $g \geq 1$
over a field $\fieldK$, and assume that $C$ has a $\fieldK$-rational
point $\Pinf$.  
Let $\Raff$ be the coordinate ring of the affine curve $C - \{\Pinf\}$;
then the group of $\fieldK$-rational points of the Jacobian of $C$ can be
identified with the ideal class group of the Dedekind domain $\Raff$.
A previous series of articles (by three different groups of authors) on
certain degree~$3$ covers of~$\Projective^1$, particularly $C_{3,4}$ curves
\cite{BEFG,BEFG-article,FO,FOR,FASKKM,OT-C35}, gives explicit formulas
for group arithmetic in the Jacobian of $C$ when $\fieldK$ is a large
finite field, under a certain genericity assumption on the divisors
whose classes are being added in the Jacobian.  This genericity
assumption was first introduced in~\cite{BEFG-article}, where such
divisors were called ``typical''; all the articles above give fast
algorithms for Jacobian  group arithmetic under the hypothesis that
the divisors (also, in fact, the pairs of divisors) one encounters are
typical, and that the result of the group operation is typical.  The
above articles, however, do not include a test to verify whether the
input divisors or output data are in fact typical, so that in
principle the algorithms might return wrong results without this being
detected during the computation.  It is desirable to know when such
algorithms have failed, so that one can redo the computation using
slightly slower algorithms that work for all divisors (for example,
those in~\cite{KKM,KKMAsymptotic}).

In this article, we give a straightforward explicit condition for a
divisor to be typical, for arbitrary $C$.  For the $C_{3,4}$ case, we show
that the algorithms in~\cite{FASKKM}, which involve two inversions (and
approximately 125 multiplications) in $\fieldK$ per group operation in the
Jacobian of a $C_{3,4}$ curve, give correct results and yield typical
divisors as output, provided that both inversions can be carried out, i.e.,
that one encounters nonzero elements of $\fieldK$ at those two moments.
Our general criterion for typicality can be expressed in terms of the
rank of certain matrices, or equivalently in terms of the structure of
suitable Gr\"obner bases for the ideal $I_D$ and its first syzygies.  For
an arbitrary curve $C$ with distinguished point $\Pinf$, we describe a
modification of the algorithms from~\cite{FASKKM}, that allows us to carry
out Jacobian arithmetic for typical elements.
To our knowledge, this is the first set of algorithms for typical
divisors on curves that (1) works with a precise definition of
typicality (and a weaker notion of semi-typicality), (2) gives necessary
and sufficient conditions on the 
typicality of the input and/or output divisors to guarantee success of
the algorithms, and (3) can certify that the end result of one's operations
is correct upon success, or else identify non-typical divisors encountered
in the computation.  Although some of the proofs are delicate, we find the
way in which the different ranks of subspaces fit together both intricate
and pleasing.

As a result of our current investigations, we discovered along the way
a nontrivial speedup of the algorithms of~\cite{FASKKM}, that saves 19
multiplications in $\fieldK$ per operation in the Jacobian, a speedup
of approximately 15\%.  This is described in the appendix to this article.


\section{Typical divisors}
\label{section2}

All divisors that we consider in this article are $\fieldK$-rational.
The reader is however encouraged to replace $\fieldK$ by its
algebraic closure $\overline{\fieldK}$, so that every divisor is a sum
of geometric 
points, without worrying about rationality.  This does not affect our
results, since everything we do boils down to the interplay between
different $\fieldK$-rational subspaces of various vector spaces, which
in turn can be recast in terms of the ranks of various matrices with
entries in $\fieldK$; these ranks are unaffected by extension of
scalars.

\begin{definition}
\label{definition2.1}
\begin{enumerate}
\item
Let $\Raff = \fieldK[C - \{\Pinf\}]$ be the affine coordinate ring of
$C$, as in the introduction, and let $N \geq 0$.
For $f \in \Raff$, we define its degree to be the order of
its pole at $\Pinf$:
\begin{equation}
\label{equation2.1.5}
\deg f = -\vinf(f);
\qquad\qquad
\text{by convention, let } \deg 0 = -\infty.
\end{equation}
\item
We define the basic Riemann-Roch space $W^N$ by
\begin{equation}
\label{equation2.1}
W^N = H^0(C,\OC(N \Pinf)) = \{f \in \Raff \mid
\deg(f) \leq N\},
\end{equation}
consisting of elements $f$ of the function field $\fieldK(C)$ that are
regular everywhere except for a pole of order at most $N$ at $\Pinf$.
Thus $\Raff = \Union_{N \geq 0} W^N$.
\item
We define a \emph{good divisor} $D$ on $C$ to be an effective
($\fieldK$-rational) divisor disjoint from $\Pinf$.
\item
For $f \in \Raff$, we will generally write $\Divisor f$ to refer to
the affine part of the divisor, unless otherwise specified.  This
means that $\Divisor f$ will completely ignore the
component at $\Pinf$, and will hence be a good divisor; we then have
$\deg (\Divisor f) = \deg f$.  The actual ``full'' divisor of $f$ is the
degree zero divisor  $\Divisor f - (\deg f)\Pinf$.
\item 
For a good divisor $D$, we define the $\fieldK$-rational 
subspace $W^N_D \subset W^N$ by
\begin{equation}
\label{equation2.2}
W^N_D = H^0(C,\OC(N \Pinf - D)) = W^N \intersect I_D,
\end{equation}
where $I_D \subset \Raff$ is the ideal of functions vanishing on $D$:
\begin{equation}
\label{equation2.3}
I_D = \{ f \in \Raff \mid \Divisor f \geq D \}.
\end{equation}
Hence $W^N_D$ is the space of elements of $I_D$ of degree at most
$N$.  We have $\dim_{\fieldK} \Raff/I_D = \deg D$.
\end{enumerate}
\end{definition}

By Riemann-Roch, we have the following dimensions:
\begin{equation}
\label{equation2.4}
\begin{split}
\dim W^N &= N + 1 - g, \text{ for } N \geq 2g - 1,\\
\dim W^N_D &= N - \deg D + 1 - g,
     \text{ for }N \geq \deg D + 2g - 1.\\
\end{split}
\end{equation}
It is moreover known that for ``most'' good divisors $D$ with $d \geq
g$, the formula for $\dim W^N_D$ above is actually
valid in the larger range $N \geq \deg D + g-1$.
We do not use this fact yet, and will come back to it later.

We also recall the notion of a base point free line bundle
$\LL$.  This means that the
sections of $H^0(C,\LL)$ do not vanish simultaneously on any
(geometric) point in $C(\overline{\fieldK})$, or equivalently on any
nontrivial $\fieldK$-rational effective divisor $E$.  In the setting
where  $\LL = \OC(N \Pinf - D)$, this means that the
only common zeros of the elements $f \in W^N_D$ occur on $D$ ---
equivalently, the elements of $W^N_D$ generate the ideal $I_D$ --- and that
moreover there exists a nonzero $f \in W^N_D$ with $\deg f = N$, to
ensure that no excess ``vanishing'' occurs at $\Pinf$.  The following
result is a standard consequence of Riemann-Roch:
\begin{equation}
\label{equation2.4.5}
\text{If } \deg \LL \geq 2g, \text{ then } \LL 
     \text{ is base point free.}  
\end{equation}
Once again, it is known that ``most'' line bundles of degree $\deg
\LL \geq g+1$ are base point free, and we will come back to
this point later in this section.  We will frequently abuse
terminology and say ``$W^N_D$ is base point free'' when we really mean
the line bundle $\OC(N \Pinf - D)$.

The following two definitions, of typical and semi-typical divisors,
are fundamental to our work.  Our immediate aim is to show that
typical divisors are semi-typical (Proposition~\ref{proposition2.5}
below), and that these notions depend only on the linear equivalence
class of the divisor (Corollary~\ref{corollary2.8} below, which
actually proves something more).  Once those results are established,
we show that all but a ``small'' set of divisors is typical
(Proposition~\ref{proposition2.8.4} below), and mention numerical
bounds (proved in~\cite{KKMbrillnoether}) for the probability that,
over a finite field, a random divisor is typical or semi-typical
(Theorem~\ref{theorem2.8.6}).

\begin{definition}
\label{definition2.4}
In the definitions below, recall that $g \geq 1$, so that $d \geq g
\geq 1$.
\begin{enumerate}
\item
A divisor $D$ is called \emph{typical} if (i) $D$ is a good divisor of
degree $d \geq g$, and (ii) there exist $s \in W^{d+g}_D$ and
$t \in W^{d+g+1}_D$ with 
$s W^{2g} + t W^{2g-1} + W^{d+g-1} = W^{d + 3g}$.
\item
A divisor $D$ is called \emph{semi-typical} if (i) $D$ is a good divisor
of degree $d \geq g$, and (ii) $W^{d+g-1}_D = 0$.
\end{enumerate}
\end{definition}

\begin{proposition}
\label{proposition2.5}
Let $D$ be a typical divisor as above.  Then:
\begin{enumerate}
\item
The sum in the definition is in fact a direct sum 
$s W^{2g} \directsum t W^{2g-1} \directsum W^{d+g-1} = W^{d + 3g}$, and we
also have a direct sum $s W^{2g} \directsum t W^{2g-1} = W^{d+3g}_D$.
\item
The divisor $D$ is semi-typical.
\item
The pair $\{s,t\}$ is an IGS (i.e., ideal generating set) for $D$,
in the terminology of~\cite{KKMAsymptotic}.  This means
in our context that the locus of common zeros of $s$ and $t$ is
precisely $D$; hence $s,t$ generate the ideal $I_D \subset \Raff$.  In
particular, $W^{d+g+1}_D$, which has degree~$g+1$, is base point free.
\item
We have that $\deg s = d+g$, $\deg t = d+g+1$, and $W^{d+g+1}_D$ is
$2$-dimensional, with basis $\{s,t\}$; any other choice of
$s',t' \in W^{d+g+1}_D$ with the same degrees is also a basis of that
space, and satisfies
$s' W^{2g} + t' W^{2g-1} + W^{d+g-1} = W^{d + 3g}$.
\end{enumerate}
\end{proposition}
\begin{proof}
The fact that the first sum is direct follows by counting dimensions in the
equality $s W^{2g} + t W^{2g-1} + W^{d+g-1} = W^{d + 3g}$.  For
example, once $s \neq 0$, we have $\dim s W^{2g} = \dim W^{2g} = g+1$,
and similarly $\dim t W^{2g-1} = g$, $\dim W^{d+g-1} = d$, and
$\dim W^{d + 3g} = d + 2g + 1$.  This implies in particular that $s,t
\neq 0$, and in fact that $s$ and $t$ are $\fieldK$-linearly
independent, since $s \in sW^{2g}$ and $t \in tW^{2g-1}$. 
By counting dimensions again, the inclusion
$s W^{2g} + t W^{2g-1} (= s W^{2g} \directsum t W^{2g-1})
     \subset W^{d+3g}_D$ must be an equality.
Now directness of the first sum implies that the intersection
$W^{d+g-1} \intersect (s W^{2g} + t W^{2g-1})$ must be zero, which can
be rewritten as $0 = W^{d+g-1} \intersect W^{d+3g}_D = W^{d+g-1}_D$,
so $D$ is semi-typical.

Since $W^{d+3g}_D$ is base point free, we deduce from
$s W^{2g} + t W^{2g-1} = W^{d+3g}_D$ that $\{s,t\}$ is an IGS
for $D$.  Also, $s$ and $t$ are linearly independent, so
$\dim W^{d+g+1}_D \geq 2$.  However, the inclusions
$0 = W^{d+g-1}_D \subset W^{d+g}_D \subset W^{d+g+1}_D$ are all of
codimension $\leq 1$ (they differ by the vanishing of one coefficient
in the Laurent expansion at $\Pinf$); hence $s$ and $t$ are indeed a
basis of $W^{d+g+1}_D$, and their degrees must
be as claimed.  Any other choice of $s',t'$ is of the form $s' = as,
t' = a't + bs$, with $a,a',b \in \fieldK$ and $a,a' \neq 0$;
the last assertion follows easily.
\end{proof}

The next two propositions describe direct sum decompositions that
occur in higher degrees, akin to the decompositions 
defining semi-typical and typical divisors.  These propositions
essentially mean that semi-typicality describes the initial ideal of
$I_D$ (in the sense of Gr\"obner bases) with respect to the term order
induced from the degree (equivalently, $\vinf$), while typicality
gives more precise information about the generators $\{s,t\}$ of $I_D$
and the relations between them.  This information is related to the
initial part of the module of first syzygies, since it describes how
both generators and relations of $I_D$ interact with the valuation at
$\Pinf$.

\begin{proposition}
\label{proposition2.6}
Let $D$ be a semi-typical divisor.  Then:
\begin{enumerate}
\item 
For all $j \geq -1$, the divisor $(d+g+j)\Pinf - D$ is nonspecial
(i.e., has trivial $H^1$).
Equivalently, by Riemann-Roch, we have $\dim W^{d+g+j}_D = j+1$.
\item
For all $j \geq -1$, we have a direct sum 
$W^{d+g+j}_D \directsum W^{d+g-1} = W^{d+g+j}$.
\end{enumerate}
Conversely, let $D$ be a good divisor with
$\deg D = d \geq g$, and suppose we know for a single
value of  $j \geq g-1$ that $W^{d+g+j}_D + W^{d+g-1} = W^{d+g+j}$
(without necessarily knowing that the sum is direct).  Then $D$ is in
fact semi-typical, and statement~(2) holds for all $j \geq -1$.
\end{proposition}
\begin{proof}
The statement on dimensions in part (1) is true for $j = -1$ by
assumption.  Thus $D' = (d+g-1)\Pinf - D$ is nonspecial.  Adding
a positive multiple of $\Pinf$ preserves the property of being
nonspecial, so we obtain part (1) for all $j$.  Part (2) follows
by counting dimensions as usual, since 
$W^{d+g+j}_D \intersect W^{d+g-1} = W^{d+g-1}_D = 0$.

As for the converse, the inequality $j \geq g-1$ ensures
that $\dim W^{d+g+j}_D = j+1$, while $\dim W^{d+g-1} = d$ and
$\dim W^{d+g+j} = d + j + 1$.  Thus the only way the inclusion
$W^{d+g+j}_D + W^{d+g-1} \subset W^{d+g+j}$ can be an equality is if
$W^{d+g+j}_D \intersect W^{d+g-1} = 0$, which implies that $D$ is
semi-typical.
\end{proof}

\begin{proposition}
\label{proposition2.7}
Let $D$ be a typical divisor, with $s,t$ as in the definition.  Then
for all $i \geq 0$, we have the direct sums
\begin{equation}
\label{equation2.7}
 s W^{2g} \directsum t W^{2g-1+i} \directsum W^{d+g-1} = W^{d + 3g + i},
\end{equation}
\begin{equation}
\label{equation2.8}
 s W^{2g+i} \directsum t W^{2g-1} \directsum W^{d+g-1} = W^{d + 3g + i},
\end{equation}
\begin{equation}
\label{equation2.9}
    s W^{2g} \directsum t W^{2g-1+i}
  = s W^{2g+i} \directsum t W^{2g-1} 
  = W^{d + 3g + i}_D.
\end{equation}
Conversely, given a good divisor $D$ with $\deg D = d \geq g$,
and elements $s \in W^{d+g}_D, t \in W^{d+g+1}_D$ such that one of
\eqref{equation2.7} or~\eqref{equation2.8} is satisfied (just with
a sum --- not necessarily with a direct sum) for a single
value of $i \geq 0$, then $D$ is typical; hence $\{s,t\}$ is an IGS for
$D$, and equations~\eqref{equation2.7}--\eqref{equation2.9} hold for all
$i$.
\end{proposition}
\begin{proof}
We know that $D$ is also semi-typical, so from part (2) of
Proposition~\ref{proposition2.6} (with $j=2g+i$) it is sufficient to
show~\eqref{equation2.9}.  We have inclusions
$s W^{2g} + t W^{2g-1+i} \subset W^{d+3g+i}_D$ and
$s W^{2g+i} + t W^{2g-1} \subset W^{d+3g+i}_D$.  The dimensions match
up correctly as though we had equality with direct sums.  To prove
equality, we need to show that $s W^{2g} \intersect t W^{2g-1+i}
= s W^{2g+i} \intersect t W^{2g-1} = 0$.  The proofs are similar, so
we will only show that the second intersection is zero.  If $u \in
W^{2g+i}$ satisfies $su \in t W^{2g-1} \subset W^{d+3g}$, then
$\deg s + \deg u \leq d+3g$.  Since $\deg s = d+g$, we
deduce that $\deg u \leq 2g$, so in fact
$su \in s W^{2g} \intersect t W^{2g-1}$, which we know to be zero
from Proposition~\ref{proposition2.5}.

Let us now prove the converse statement.  Suppose for one $i$
that~\eqref{equation2.7}, say, holds (the proof in the case
of~\eqref{equation2.8} is similar).  Then, by counting dimensions as
in part (1) of the proof of Proposition~\ref{proposition2.5}, we
obtain that the sum is direct, and that any sum made from subspaces of
$s W^{2g}$, $t W^{2g-1+i}$, and $W^{d+g-1}$ will remain direct.  Thus
$sW^{2g} + tW^{2g-1} + W^{d+g-1}
   = sW^{2g} \directsum tW^{2g-1} \directsum W^{d+g-1}
 \subset W^{d+3g}$,
and we have equality by comparing dimensions.  Thus $D$ is typical.
\end{proof}

\begin{corollary}
\label{corollary2.8}
Whether a good divisor $D$ with $\deg D = d \geq g$
is typical or semi-typical depends only on the divisor class
$[D - d \Pinf] \in \Pic^0(C)$; in particular, it depends only on the class
$[D] \in \Pic^d C$.
\end{corollary}
\begin{proof}
Suppose $D'$ is another good divisor of degree
$d' \geq g$ that maps to the same element of $\Pic^0(C)$.  This means that
there exists a nonzero element $u \in \fieldK(C)$ of the function field
whose full divisor (including $\Pinf$) is
$\Divisor u = D' - D + (d-d')\Pinf$.  It then follows that 
$W^{d'+j}_{D'} = uW^{d+j}_D$ for all $j$, so (taking $j=g-1$) we see that
$D'$ is semi-typical if and only if $D$ is. 

Now suppose that $D$ is typical, and let $s,t$ be as in the definition.
Define $s' = us \in W^{d'+g}_{D'}$ and $t' = ut \in W^{d'+g+1}_{D'}$.  We
have $s' W^{2g} + t' W^{2g-1} = u W^{d+3g}_D$ by
Proposition~\ref{proposition2.5}, and this last space is equal to
$W^{d'+3g}_{D'}$.  We also know that $D'$ is semi-typical,
by the preceding paragraph, so $W^{d'+3g}_{D'}+W^{d'+g-1} = W^{d'+3g}$ by
Proposition~\ref{proposition2.6}.  Hence $s' W^{2g} + t' W^{2g-1} +
W^{d'+g-1} = W^{d'+3g}$, so $D'$ is also typical, as desired.
\end{proof}

We mention for completeness a characterization of typical
divisors, phrased in terms of the degree zero line bundle
$\LL = \OC(d\Pinf-D)$.  The following is Proposition~3.2
of~\cite{KKMbrillnoether}, and is proved by techniques similar to those of
Proposition~\ref{proposition2.8.4} in this article.
Note that the first two conditions below essentially say that
$\LL$ and $\LL^{-1}$ are semi-typical.

\begin{proposition}
\label{proposition2.8.1}
With the above notation, $D$ is typical if and only if the following three
conditions hold:
(i) $H^0(C, \LL((g-1)\Pinf)) = 0$,
(ii) $H^0(C, \LL^{-1}((g-1) \Pinf)) = 0$, and
(iii) $\LL((g+1)\Pinf)$ is base point free.
\end{proposition}

We now turn to the sense in which ``most'' divisors are typical (hence
also semi-typical).  More precisely, we consider divisor classes
(equivalently, isomorphism classes of line bundles) in $\Pic^d C$; due
to the presence of the rational point $\Pinf$, we can identify $\Pic^d C$
with $\Pic^0 C$, and hence with the $g$-dimensional Jacobian variety of
$C$.  In this setting, a statement about ``most'' divisor classes or
line bundles will mean a statement that holds for all elements of
$\Pic^d C$ outside a finite union of ``bad'' subvarieties of dimension
$\leq g-1$.  This can all be studied over $\overline{\fieldK}$,
without changing the dimension of the appropriate bad subvarieties.

Moreover, at the expense of eliminating a few more subvarieties of
dimension $g-1$, we can restrict, in the case $d \geq g$, to classes of
good divisors.  The way to see this is to fix a good divisor $E$, with
$\deg E = d-g$, and then to represent any divisor class, identified
with a line bundle $\LL$ of degree $d$, by an effective divisor 
$D' + E$, with $\deg D' = g$; this is possible because
$H^0(C,\LL(-E)) \neq 0$.  Then we eliminate from consideration the
divisor classes represented by $E + D'$ as $D'$ varies over divisors
of the form $\Pinf + Q_1 + \cdots + Q_{g-1}$, which describe a
$(g-1)$-dimensional subset of $\Pic^d C$ as the $Q_i$ vary.  Extensions
and variations of this argument give rise to the following results,
which we have referred to earlier, and which we will use to control
the size of the locus of nontypical divisors.

\begin{lemma}
\label{lemma2.8.2}
In the following statements, ``most elements $\LL \in \Pic^d C$''
refers to all but a finite union of at most $(g-1)$-dimensional
subvarieties, as discussed above.
\begin{enumerate}
\item
If $d \leq g-1$, then most elements $\LL \in \Pic^d C$ satisfy
$H^0(C,\LL) = 0$.
\item
Most elements $\LL \in \Pic^{g+1} C$ are base
point free,  with $\dim H^0(C,\LL) = 2$.
\end{enumerate}
\end{lemma}
\begin{proof}
We include the proof of this standard result for completeness.  As
noted above, we can work over $\overline{\fieldK}$.  If $d \leq g-1$,
then an $\LL$ with nonzero $H^0$ must contain an effective divisor in
its class, i.e., $\LL \isomorphic \OC(Q_1 + \cdots + Q_d)$, which
varies in a $d$-dimensional subvariety as the $Q_i$ vary.  This proves
statement~(1) above.

As for statement~(2), it follows from Riemann-Roch that every
$\LL \in \Pic^{g+1} C$ has $\dim H^0(C,\LL) \geq 2$.  If such an $\LL$
is not base point free, there exists $P \in C(\overline{\fieldK})$
such that $\dim H^0(C,\LL(-P)) = \dim H^0(C,\LL) \geq 2$; hence $\dim
H^1(C, \LL(-P)) \geq 1$.  Writing $\omega$ for the 
canonical bundle, we conclude that the degree $g-2$ line bundle 
$\omega \tensor \LL^{-1} (P)$ has nontrivial $H^0$, hence that
$\omega \tensor \LL^{-1} (P) \isomorphic
 \OC(Q_1 + \cdots + Q_{g-2})$.
Hence $\LL \isomorphic \omega(P - Q_1 - \cdots - Q_{g-2})$, and the
family of such $\LL$ has dimension $g-1$ as $P$ and the $Q_i$ vary.
Moreover, if $\dim H^0(C,\LL) \geq 3$, then for every $P \in
C(\overline{\fieldK})$, we again have $\dim H^0(C,\LL(-P)) \geq 2$,
which cannot happen for most $\LL$ (even for one $P$ depending on
$\LL$, as we have just seen).

We note that the above reasoning generalizes to show that if $d \geq
g+1$, then most choices of $\LL \in \Pic^d C$ are base point free, with
$\dim H^0(C, \LL) = d + 1 - g$.
\end{proof}

\begin{proposition}
\label{proposition2.8.4}
Let $d \geq g \geq 1$.  Then most elements $\LL \in \Pic^d C$ are typical
and hence also semi-typical.  The word ``most'' here has the same
meaning as in Lemma~\ref{lemma2.8.2}.
\end{proposition}
\begin{proof}
As remarked earlier, we can restrict to the situation when $\LL = \OC(D)$,
with $D$ a good divisor.  It is convenient to show first that $\LL$ is
semi-typical, even though this is implied by the full result.  Indeed,
as $\LL$ varies in $\Pic^d C$, the line bundle $\OC((d+g-1)\Pinf - D)
= \LL^{-1}((d+g-1)\Pinf)$ varies in $\Pic^{g-1} C$.  Thus (for most
$\LL)$ it has no global sections, by Part~(1) of
Lemma~\ref{lemma2.8.2}, and hence $D$ is semi-typical.

We now use Part~(2) of the above lemma to conclude that, for most $D$,
the space $W^{d+g+1}_D$ is two-dimensional and base point free. Thus
any basis $\{s,t\}$ of $W^{d+g+1}_D$ is an IGS.  We can control the
degrees of $s$ and $t$ so as to obtain $s,t$ as in
Definition~\ref{definition2.4}.  Indeed, we already have $W^{d+g-1}_D = 0$,
so necessarily $\dim W^{d+g}_D = 1$, and we can choose (any nonzero)
$s \in W^{d+g}_D  \subsetneq W^{d+g+1}_D$, and $t \in W^{d+g+1}_D$,
with $t \notin W^{d+g}_D$.  It then follows that $\deg s = d+g$ and
$\deg t = d+g+1$, and we have obtained our desired $\{s,t\}$ which is
an IGS for $W^{d+g+1}_D$.

It is immediate that $sW^{2g} + tW^{2g-1} \subset
W^{d+3g}_D$, and our next goal is to show that the above inclusion is
an equality for most $D$.  This is similar to the proof of
Lemma~4.10 in~\cite{KKMAsymptotic}, and is in essence the base point
free pencil trick.  We have as usual $\dim sW^{2g} = \dim W^{2g} =
g+1$, and similarly $\dim tW^{2g-1} = g$, while $\dim W^{d+3g}_D = 2g
+ 1$; it thus suffices to show that $sW^{2g} \intersect tW^{2g-1} = 0$.
Write $\Divisor s = D + A$ and $\Divisor t = D + B$, where $A$ and $B$
are effective divisors of degrees $g$ and $g+1$, respectively (recall
that $\Divisor s$ and $\Divisor t$ ignore the poles at $\Pinf$).
Since $\{s,t\}$ is an IGS, $A$ and  $B$ are disjoint.  We thus have
$sW^{2g} \intersect tW^{2g-1} = 
W^{d+3g}_{D+A} \intersect W^{d+3g}_{D+B} = W^{d+3g}_{D+A+B}$.  Via
division by the product $st$, which introduces additional poles at
$2D + A + B$, we have that $W^{d+3g}_{D+A+B}$ is isomorphic
to $H^0(\OC((g-d-1)\Pinf + D))$.  But this last space is the divisor
of a line bundle of degree~$g-1$, and hence is zero for most $D$.
Hence $sW^{2g} + tW^{2g-1} = W^{d+3g}_D$, as desired.

At this point, we know that $sW^{2g} + tW^{2g-1} + W^{d+g-1} = 
W^{d + 3g}_D + W^{d+g-1} \subset W^{d + 3g}$, and we wish to show
equality to conclude that $D$ is typical.
As usual, we count dimensions: the space $W^{d+3g}_D$ has codimension $d$
inside $W^{d+3g}$, and $\dim W^{d+g-1} = d$, so it suffices to point
out that $W^{d+3g}_D \intersect W^{d+g-1} = W^{d+g-1}_D$, which is
zero since we already know that $D$ is semi-typical.
\end{proof}


When $\fieldK$ is a finite field with $q$ elements, one can give a precise
quantitative bound of the probability that a random divisor class fails to
be typical or semi-typical.  Qualitatively, the expected probability is
$O(1/q)$ (for fixed $g$), since the nontypical divisors lie on a proper
subvariety, but finding the implied constant takes some work.  The
following result is Theorem~3.3 of~\cite{KKMbrillnoether}.  The proof 
there uses rather different techniques from the ones in this article, based
on bounding the number of points on certain Brill-Noether loci via the Weil
bounds for zeta and L-functions of curves.  The probabilities given below
are very small for the values of $g$ and $q$ one is likely to
encounter in applying the results of this article; see 
Remark~\ref{remark4.2} below. 

\begin{theorem}
\label{theorem2.8.6}
Suppose that $g\geq 2$ and that $\fieldK$ is a finite field with $q$
elements, with $q \geq 16^g$.  Let $\LL$ be a uniformly randomly
chosen element of $\Pic^d C$
(the precise value of $d$ does not matter, since we can always shift by a
multiple of $\Pinf$, as in Corollary~\ref{corollary2.8}).
\begin{enumerate}
\item
  The probability that $\LL$ is not semi-typical is at most $1.7/q$.
\item
  The probability that $\LL$ is not typical is at most
  $(16^g \cdot g + 3.4)/q$.
\item
  The probability that $\LL$ and $\LL^{-1}$ (more accurately,
  replace $\LL^{-1}$ by $\LL^{-1}(N\Pinf)$ for some large $N$) are not
  both typical is at most $(16^g \cdot 2g + 3.4)/q$.
\end{enumerate}
\end{theorem}

\begin{example}
\label{example2.8.8}
We illustrate semi-typical and typical divisors in the setting of $C_{3,4}$
curves, and relate our definitions to the constructions in~\cite{FASKKM}.
Recall that a $C_{3,4}$ curve $C$ has genus $g=3$, and is given by an
affine equation for the open set $C - \{\Pinf\}$ of the form
$f(x,y) = y^3 - x^4 + \sum_{3i+4j < 12} c_{ij} x^i y^j = 0$; here
$\Raff = \fieldK[x,y]/\langle f(x,y) \rangle$, with $\deg x = 3$ and
$\deg y = 4$.  Any given $W^N$ is spanned by the first few monomials
from the ordered list $1, x, y, x^2, xy, y^2, x^3, x^2 y, xy^2, \dots$
of degrees $0, 3, 4, 6, 7, 8, 9, 10, 11, \dots$; the
next monomial, of degree $12$, can be either $y^3$ or $x^4$, and for
each larger $N$ we continue the list at the $N$th step by choosing
once and for all a monomial $x^i y^j$ with $3i + 4j = N$, for
example by limiting to $j \leq 2$.  The resulting monomials
$x^{2+i}, x^{1+i}y, x^i y^2, \dots$ for $i \geq 0$ give elements of
$\Raff$ of all degrees starting with $2g = 6$.
The dimensions of $W^N$ for $N=0,1,2,3,\dots$ are respectively
$1,1,1,2,3,3,4,5,6,7,\dots$; for example, $W^4 = W^5
 = \fieldK \cdot 1 + \fieldK \cdot x + \fieldK \cdot y$, and for 
$N \geq 5 = 2g - 1$, $\dim W^N = N-2$, illustrating the first line
of~\eqref{equation2.4}.

On such a curve, let $D$ be a good divisor of degree~$3$.  In
Proposition~2.1 and Equation~(2) on page~310 of~\cite{FASKKM}, we
asserted that the ideal $I_D$ was ``typically'' 
generated by two elements
$F = x^2 + ay + bx + c \in W^6_D$ and $G = xy + dy + ex + f \in W^7_D$,
which play the roles of $s,t$ in Definition~\ref{definition2.4}.
We also asserted that ``typically'' $a \neq 0$.
Proposition~\ref{proposition2.9} below shows
that the existence of $F,G$ as above, with $a \neq 0$, is precisely
equivalent to having $D$ be typical, according to the definition in this
article.  The 
idea is that with $F,G$ as above, the ideal $I_D$ contains an element 
$H = a^{-1}(yF - xG) = y^2 + \cdots \in W^8_D$.  It follows that 
$I_D$ contains the list of elements $F, G, H, xF, xG, xH, x^2 F,\dots$,
of degrees $6, 7, 8, 9, \dots$, where $6 = d+g$;
moreover, $W^5_D = 0$ (accepting for the moment that such a $D$ is
semi-typical),  while $W^N_D$ for $N \geq 6$ has as a basis the 
first $N-5$ elements of $I_D$ from the above list.
This illustrates the second line of~\eqref{equation2.4}.
\end{example}

Here is our characterization of typical and semi-typical
divisors on $C_{3,4}$ curves.

\begin{proposition}
\label{proposition2.9}
Let $D$ be a good divisor of degree $3$ on a $C_{3,4}$ curve.
\begin{enumerate}
\item
$D$ is semi-typical if and only if there exist elements
$F,G,H \in I_D$ of degrees $6,7,8$.  After rescaling $F,G,H$ by
nonzero elements of $\fieldK$ (to make them ``monic''), and possibly
replacing $G$ by $G-\lambda F$ for some $\lambda \in \fieldK$ (to
eliminate the $x^2$ term), we can assume that 
$F = x^2 + ay + bx + c \in W^6_D$, $G = xy + dy + ex + f \in W^7_D$, and
$H = y^2 + \cdots \in W^8_D$.
\item
$D$ is typical if and only if there exist elements $F,G$ as above,
with $a \ne 0$.  In that case, we can take $H = a^{-1}(yF - xG)$.
\end{enumerate}
\end{proposition}
\begin{proof}
Semi-typicality implies the existence of $F,G,H$ as above because of
our control over the dimensions of $W^N_D$ in part~(1) of
Proposition~\ref{proposition2.6}.  Conversely, the existence 
of $F,G,H$ ensures that $W^8_D + W^5$ contains elements with leading
terms $y^2, xy, x^2, y, x, 1$, hence a ``triangular'' basis for $W^8$.
We can hence apply the converse statement in
Proposition~\ref{proposition2.6}.

As for typicality, take the following bases for $F W^6$, $G W^5$, and
$W^5$, respectively: 
$\{F, xF, yF, x^2F\}$, $\{G, xG, yG\}$, and $\{1, x, y\}$.
Reordering all these elements and performing a harmless ``elementary
operation'' in linear algebra, we see that the subspace
$F W^6 + G W^5 + W^5 \subset W^{12}$ is spanned by the set of elements 
$\{1, x, y, F, G, (yF-xG), xF, xG, yG, x^2 F\}$.   The leading terms
of these elements are respectively
$1, x, y, x^2, xy, ay^2, x^3, x^2 y, xy^2, x^4$,
and hence our set is ``triangular'' in terms of the obvious basis for
$W^{12}$.  Writing this in matrix form, we obtain a triangular matrix
with almost all diagonal entries equal to~$1$, except for a single
diagonal entry of~$a$ in the column corresponding to 
$yF-xG = ay^2 + \cdots$; thus our set generates all of $W^{12}$ if and
only if $a \neq 0$.  
\end{proof}


\section{Operations on typical and semi-typical divisors}
\label{section3}

We now investigate how typicality and semi-typicality allow us to
describe ``generically correct'' algorithms for operations on
divisors, that succeed, roughly speaking, when the input and/or output
is typical (or sometimes semi-typical), and certify both success and
(semi-)typicality of the input and/or output.  As a running example,
we illustrate our general constructions in the setting of of $C_{3,4}$
curves, thereby making the connection with~\cite{FASKKM}.

We adopt the following conventions in this section:
\begin{enumerate}
\item
All letters describing divisors (such as $D$, $D'$, $E$, etc.) refer to
good divisors, unless otherwise specified;
\item
We will also use the corresponding lowercase letter to refer to the
degree of the divisor, so $\deg D = d$, $\deg D' = d'$, $\deg E = e$,
etc.;
\item
We will always assume that these degrees are $\geq g$, and that
$g \geq 1$.
\end{enumerate}

The context in which we will later use such divisors is that a $D$ of
the above type represents the divisor class
$[D - d \Pinf] \in \Pic^0(C)$.  Conversely, every element of
$\Pic^0(C)$ can be written as $[D - g \Pinf]$ for some effective
$\fieldK$-rational divisor $D$ of degree $g$, which a priori may have
$\Pinf$ in its support, i.e., not be good.  When $D$ is good, it is well
known (and basically tautological) that the choice of such a $D$ is
unique when $D$ is reduced; in Proposition~\ref{proposition3.4.5}, we
recall the definition of a reduced divisor, and later show in
Corollary~\ref{corollary3.5.5} that typical divisors are always
reduced.  Hence 
typical elements of $\Pic^0(C)$ have a unique representation by a good
divisor $D$ of degree~$g$, and we do not need to go through the more
elaborate tests for equality used in the general algorithms
of~\cite{KKM,KKMAsymptotic}.  We still need divisors of degrees
$d \geq g$ to represent various intermediate results in our algorithms,
so we carry out the discussion below for general $d$.


\subsection{Addition of two divisors by taking an intersection}
\label{subsection3.1}

The first question we study concerns taking the intersection of two spaces
$W^N_D$ and $W^N_{D'}$.  
In most cases, one expects $D$ and $D'$ to be disjoint, so the
intersection should be the space $W^N_{D+D'}$, or at least our
calculation should be able to detect when this is indeed the case.
The following result is straightforward.

\begin{proposition}
\label{proposition3.1}
Suppose given spaces $W^N_D$ and $W^N_{D'}$ as above, viewed as subspaces
of $W^N$, and suppose that $N \geq d+d'+g-1$.  Compute the 
intersection $W^N_D \intersect W^N_{D'}$ in a way that simultaneously 
yields the subspace 
$\hat{W} = W^{d+d'+g-1}_D \intersect W^{d+d'+g-1}_{D'}$.
If $\hat{W} = 0$, then in fact $D$ and $D'$ are disjoint, the
intersection above correctly computes $W^N_{D+D'}$, and $D+D'$ is
semi-typical.

Conversely, if $D$ and $D'$ are disjoint, and $D+D'$ is semi-typical,
then the subspace $\hat{W}$ will indeed be zero, and hence
$W^N_D \intersect W^N_{D'} = W^N_{D+D'}$, as desired.
\end{proposition}
\begin{proof}
Write $E = \text{lcm}(D,D')$; thus $W^n_D \intersect W^n_{D'} = W^n_E$ for
all values of $n$, including both $n=N$ and $n=d+d'+g-1$.  The fact that
$\hat{W} = W^{d+d'+g-1}_E = 0$ forces $e = \deg E \geq d+d'$, from
which we deduce 
that $E = D+D'$ and that $D$ and $D'$ are disjoint.  The first result
follows.  As for the converse, disjointness of $D$ and $D'$ means that
$E = D+D'$.  This divisor is semi-typical, so $\hat{W} = 0$.
\end{proof}

We now discuss how one can effectively carry out linear algebra
computations in subspaces of $W^N$, such as computing the intersections in
the above proposition; this generalizes the presentation in~\cite{FASKKM}.
Elements of $W^N$, for sufficiently large $N$, are 
represented as column vectors in $\fieldK^{N+1-g}$ with respect to some
basis of ``monomials'' in $\Raff$, ordered by degree.
A subspace such as $W^N_D \subset W^N$ is represented as a matrix
whose columns form a basis for $W^N_D$.  When possible, we convert the
basis to column-echelon form, so that the columns represent a basis
for $W^N_D$ in order of increasing degree; this is illustrated in
equation~\eqref{equation3.2} in Example~\ref{example3.2} below. 

For computing the intersection, it is useful to set up a specific
isomorphism between the quotient $\Raff/I_D$ (which was called
$\mathcal{A}$ in Section~3 of~\cite{FASKKM}) and the vector space
$W^{d+g-1}$.  Since $D$ is semi-typical, we know that for $N \geq
d+g-1$, the subspace $W^N \subset \Raff$ surjects onto $\Raff/I_D$, 
with kernel $W^N_D$.  Moreover, $W^N = W^N_D \directsum W^{d+g-1}$, so we
can therefore identify $\Raff/I_D$ with $W^{d+g-1}$, and we have a
compatible family of vector space isomorphisms 
\begin{equation}
\label{equation3.1}
   W^N/W^N_D \xrightarrow{\isomorphic} \Raff/I_D \isomorphic W^{d+g-1},
\qquad \text{for all } N \geq d+g-1.
\end{equation}
Concretely, the composition $W^N \to W^N/W^N_D \isomorphic W^{d+g-1}$
amounts to taking elements of $W^N$, viewed as column vectors, and reducing
the columns with respect to the columns of the matrix describing a basis
for $W^N_D$ mentioned above.  This reduces everything to an element of
$W^{d+g-1}$, i.e., to a vector in $\fieldK^d$.

Let us denote by $r: W^N \to W^{d+g-1}$ the resulting reduction map mod $D$.
We can now compute the intersection $W^N_D \intersect W^N_{D'}$ as the
kernel of the composite map $W^N_{D'} \hookrightarrow W^N \xrightarrow{r}
W^{d+g-1}$.  This composite map can be represented by a $d \times
(N-d'-g+1)$ matrix, which we shall call $M$; an equivalent matrix
is called $M'$ in Section~6 of~\cite{FASKKM}.  One can compute $M$ as
the product of the matrix for $r$ by the matrix whose columns give a
basis for $W^N_{D'}$; alternatively, take the matrix for $W^N_{D'}$,
and reduce each column (modulo $W^N_D$) to obtain columns describing
the corresponding images (i.e., the reductions) in $W^{d+g-1}$.

\begin{example}
\label{example3.2}
We illustrate the above in the $C_{3,4}$ case.  The basis of ``monomials''
begins with $1,x,y,x^2,\dots$, as we saw in Example~\ref{example2.8.8}.
Let $D$ be semi-typical of degree $3$, with 
elements $F,G,H$ as in Proposition~\ref{proposition2.9}.  The columns
of the matrix representing $W_D^N$ will then encode the echelon basis
$F,G,H,xF,xG,xH,x^2 F,\dots$.  For example, when $N=10$, then the
basis of $W^{10}_D$ is $\{F,G,H,xF,xG\}$.  Write $F=x^2+ay+bx+c$,
$G=xy+dy+ex+f$, and $H = y^2+pxy+qx^2+ry+sx+t$ for certain
$p,q,\dots,t\in \fieldK$.  (In the typical case, $H = a^{-1}(yF-xG)$,
so one can express $p, q, \dots, t$ in terms of $a,b,\dots,f$.)
We thus obtain the following echelon form matrix which describes $W^{10}_D$:
\begin{equation}
\label{equation3.2}
\begin{pmatrix}
c&f&t&0&0\\  
b&e&s&c&f\\  
a&d&r&0&0\\  
1&0&q&b&e\\  
0&1&p&a&d\\  
0&0&1&0&0\\  
0&0&0&1&0\\  
0&0&0&0&1    
\end{pmatrix},
\qquad
\text{each row representing the coefficient of }
\begin{matrix}
1\\
x\\
y\\
x^2\\
xy\\
y^2\\
x^3\\
x^2 y
\end{matrix}.
\end{equation}
The first three columns of the above matrix describe of course
$W^8_D$. 

(For general curves, with $D$ semi-typical, we would take
generators for $W^N_D$ of degrees $d+g, d+g+1, \dots, N$, provided
$N \geq d+g-1$; the extreme case $N=d+g-1$ would correspond to an
empty matrix.  The columns of the matrix have their lowest nonzero
entries  in rows $d+1, d+2, \dots, N+1-g$, because row $d+1$
corresponds to a degree $d+g$ element in the $(d+1)$-dimensional space
$W^{d+g}$.)

We now illustrate the computation of an intersection as in
Proposition~\ref{proposition3.1}.  We first explain the reduction map
$r: W^{10} \to W^5$, where $d+g-1 = 3+3-1 = 5$ in our setting, and
the $3$-dimensional space $W^5$ (which equals $W^4$ here) has basis
$\{1,x,y\}$.  Given an 
element of $W^{10}$, represented by a column vector $v \in \fieldK^8$,
we obtain its reduction by successively subtracting from $v$ multiples
of the columns of the matrix in~\eqref{equation3.2}, from the rightmost
column to the leftmost, in order to eliminate the lowest entries of
$v$ from the bottom up.  One is left with a reduced vector with only
three possibly nonzero entries at the top.  We identify the corresponding
element of $W^5$ with the column vector in $\fieldK^3$ consisting of these
top three entries.

To compute the intersection $W^{10}_D \intersect W^{10}_{D'}$,
suppose given the analogous matrix for $W^{10}_{D'}$, in terms of the
coefficients of $F' = x^2 + a'y + b'x + c'$, $G'=y^2 + d'y + \cdots$,
and so forth.  The matrix $M$ defined above is obtained by reducing
the columns of the matrix for $W^{10}_{D'}$.  Thus the columns of $M$ 
give the reductions of $F',G',H',xF',xG'$ modulo the columns of the
matrix for $W^{10}_D$ in~\eqref{equation3.2}.  This yields
\begin{equation}
\label{equation3.3}
M = 
\begin{pmatrix}
c'-c & f'-f & * & * & *\\
b'-b & e'-e & * & * & *\\
a'-a & d'-d & * & * & *
\end{pmatrix}.
\end{equation}
The first two columns describe the reductions $F'-F, G'-G \in W^5$ 
of $F'$ and $G'$ modulo $W^{10}_D$.  (These were called $B_{F'},
B_{G'}$ in Section~4 of~\cite{FASKKM}.)  One obtains the third column, for
instance, by reducing the column representing $H'$ to the column
corresponding to $H'-H-(p'-p)G-(q'-q)F\in W^5$, which is a reduction with
respect to the first three columns of the matrix
in~\eqref{equation3.2}.  The last two columns are similar.

Now the kernel of $M$ corresponds 
to linear combinations of $F',G',H',xF',xG'$
that belong to $W^{10}_D \intersect W^{10}_{D'}$.
As seen in Section~6 of~\cite{FASKKM}, the algorithms there
find $\ker M$ by a Gaussian elimination that assumes that the leftmost
$3\times3$ submatrix of $M$ is invertible.  This amounts to
invertibility of the leftmost $3\times 3$ minor, called $U$ in
equation~(15) of~\cite{FASKKM}\footnote{Actually, the algorithm there also
assumes that the top left $1\times1$ and $2\times 2$ minors, $A_1$ and
$D = \Delta_{12}$, are also invertible, and replaces inverting all
three quantities $A_1, D, U$ by one field inversion combined with several
multiplications.  To genuinely only compute $\ker M$ while checking
that $U \neq 0$, one can exchange rows of $M$ as needed,
which does not change the kernel or the fact that $U\neq 0$.  Thus,
possibly after a first row exchange, one can first ensure that
$A_1 \neq 0$, then one computes 
$\Delta_{12}$ and $\Delta_{13}$, which are both needed anyhow for the
computation.  One then exchanges rows 2 and~3 if 
needed to ensure that $\Delta_{12} \neq 0$.}.
The key point to observe is that the above $3\times3$
leftmost submatrix of $M$, with determinant~$U$, represents a matrix
whose kernel computes $W^8_D \intersect W^8_{D'}$.  Thus invertibility
of this submatrix means 
that $W^8_D \intersect W^8_{D'} = 0$, and that we satisfy the
condition of Proposition~\ref{proposition3.1}.  Hence the computation,
if successful,  returns the correct result for $W^{10}_{D+D'}$ in
terms of $\ker M$; the way in which this kernel is computed, which
essentially expresses the fourth and fifth columns of $M$ as linear
combinations of the first three columns, simultaneously ends up computing
``monic'' elements $s \in W^9_{D+D'}, t\in W^{10}_{D+D'}$.
\end{example}

We have just shown the following result.

\begin{proposition}
\label{proposition3.3}
In the $C_{3,4}$ case, let $D$ and $D'$ be typical divisors of
degree~$3$.  Compute $s \in W^9, t \in W^{10}$ as in Sections 3, 4, 6,
and~7 of~\cite{FASKKM}.  If the inversion in Proposition~6.1
of~\cite{FASKKM} can be carried out\footnote{Possibly allowing as
before for row operations, so the only condition that really gets
checked is $U \neq 0$.},
then the result correctly
produces $s,t \in I_{D+D'}$, and one deduces that $D,D'$ were disjoint
to begin with and that $D+D'$ is semi-typical.  The converse also holds.
\end{proposition}

The same argument as in Proposition~\ref{proposition3.1} and
Example~\ref{example3.2} generalizes to show:

\begin{proposition}
\label{proposition3.4}
Let $C$ be arbitrary, and consider typical divisors $D,D'$ of degrees
$d,d'$ (both $\geq g$, as usual), described by elements
$F \in W^{d+g}_D, G \in W^{d+g+1}_D$ and 
$F' \in W^{d'+g}_{D'}, G' \in W^{d'+g+1}_{D'}$.
For $N \geq d+d'+g-1$, the following algorithm will either fail or
succeed, and, if it succeeds, will correctly compute $W^N_{D+D'}$.
The algorithm succeeds if and only if (i) $D,D'$ were disjoint to
begin with, and (ii) $D+D'$ is semi-typical.

\smallskip 
\noindent  
\emph{Algorithm:} 
\begin{enumerate}
\item 
Compute column-echelon matrices whose columns represent bases for
$W^N_D$ and $W^N_{D'}$, respectively.  For example, if $N = d+3g+i$
with $i \geq 0$, one can start with a basis for $W^N_D$ obtained from
$FW^{2g+i} \directsum GW^{2g-1}$ as in~\eqref{equation2.9}, write the
basis as columns, and then ``column reduce'' the resulting matrix into
echelon form; if $N < d+3g$, one can compute the column-reduced matrix
for the larger space $W^{d+3g}_D$, and select the first $N-d-g+1$
columns.
(Note:  If $D,D'$ are merely semi-typical, but we have access
        nonetheless to column-echelon bases for the spaces $W^N_D$
        and $W^N_{D'}$, then we can still use these spaces and proceed
        to the next step.)
\item
Using the matrix for $W^N_D$, reduce the columns coming from the
matrix for $W^N_{D'}$ to representatives in $W^{d+g-1} \isomorphic
W^N/W^N_D$.  This yields a matrix $M$ of size $d \times (N-d'-g+1)$,
whose columns represent the reduction modulo $I_D$ of the basis of
$W^N_{D'}$, ordered by increasing degree (of the original basis
element, not of the reduction).
\item
If the leftmost $d\times d$ submatrix of $M$ is not invertible (easily
seen during the linear algebra, e.g., by carrying out Gaussian
elimination), then return ``fail''.  This is because the leftmost $d$
columns of $M$ represent the map from $W^{d+d'+g-1}_{D'}$ to
$\Raff/I_D \isomorphic W^{d+d'+g-1}/W^{d+d'+g-1}_D \isomorphic W^{d+g-1}$.
\item
Otherwise, compute an echelon basis for the kernel of $M$.  This will
consist of column vectors of the form
$(*,\dots, 1, 0, \dots)^{\mathbf{T}}$  
with at least $d$ initial entries before the final nonzero entry $1$.
By taking the corresponding linear combinations of the previous
ordered basis for $W^N_{D'}$, convert the basis for $\ker M$
into a basis for $W^N_{D+D'}$, ordered by increasing degree.
Return ``succeed'', along with the basis for $W^N_{D+D'}$. 
\end{enumerate}
\end{proposition}

\subsection{Flipping a divisor}
\label{subsection3.2}
We now turn to the question of finding a complementary divisor for a
given semi-typical divisor $D$.  Consider a nonzero element 
$s\in W^{d+g}_D$, which is unique up to a multiplicative constant.  Hence
the divisor of $s$ is uniquely determined, and we have (ignoring as
usual poles at $\Pinf$) that $\Divisor s = D+A$, with $\deg A = g$.
The divisor $A$ is complementary to $D$, and our goal is to compute
the space $W^N_A$ for  suitable $N$; we shall refer to this operation
as ``flipping'' the divisor $D$.  On the level of ideals of $\Raff$,
the effect of flipping is to compute the colon ideal
$I_A = (s\Raff : I_D)$, which satisfies $I_D \cdot I_A = s\Raff$.
In the Jacobian, this corresponds to replacing the class
$[D - d\Pinf] \in  \Pic^0(C)$ by its negative class $[A - g\Pinf]$,
since the ``full'' principal divisor of~$s$ is 
$\Divisor s = D+A - (d+g)\Pinf$.

The following proposition shows that the result $A$ of flipping is
a reduced divisor; thus flipping combines inverting the
class of $D$ (or, more precisely, of $D-d\Pinf$) in the Jacobian, and
reducing the result.

\begin{proposition}
\label{proposition3.4.5}
If $D$ is semi-typical, then its flip $A$ is reduced along $\Pinf$,
meaning that $A$ is not linearly equivalent to any divisor of the form
$A' + \Pinf$, with $A'$ effective.
\end{proposition}
\begin{proof}
If $A$ were equivalent to $A' + \Pinf$, then $D+A' - (d+g-1)\Pinf$
would be principal, so there would exist a nonzero element $s' \in
W^{d+g-1}_D$, contradicting the semi-typicality of $D$.
\end{proof}

We can compute $W^N_A$ similarly to Subsection~2.2 of~\cite{FASKKM}.
Take a nonzero $t \in W^{d+g+1}_D$, with $\deg t = d+g+1$.  Write
$\Divisor t = D+B$, with $\deg B = g+1$.  In the typical case, we know
that $s\Raff + t\Raff = I_D$, since $\{s,t\}$ form an IGS for $D$;
equivalently, $A$ and $B$ are disjoint.  In that case, we can compute
\begin{equation}
\label{equation3.4}
\begin{split}
W^N_A &= \{ \ell \in W^N \mid \ell I_D \subset s\Raff\}
       = \{ \ell \in W^N \mid \ell s, \ell t \in s\Raff \} \\
      &= \{ \ell \in W^N \mid \ell t \in s W^{N+1} \},
\qquad \text{ assuming } \{s,t\} \text{ an IGS for } D.
\end{split}
\end{equation}
As in the discussion preceding Lemma~2.4 of~\cite{FASKKM}, setting up
a system of linear equations to solve~\eqref{equation3.4} is wasteful.
Indeed, such a system essentially computes $tW^N \intersect sW^{N+1}$
inside the overly large space $W^{N+d+g+1}$, even though both subspaces lie
inside the smaller space $W^{N+d+g+1}_D$, which is usually of codimension
$d$ inside $W^{N+d+g+1}$.  To remove the excess dimensions from
consideration, we proposed in~\cite{FASKKM} to carry out a
``truncated'' intersection after projecting to the quotient
$W^{N+d+g+1}/W^{d+g-1}$.  This truncation amounts computationally to
ignoring the top $d$ rows of the matrix whose kernel describes the
intersection in~\eqref{equation3.4}, as in Section~8 of~\cite{FASKKM}.
We can also describe this truncated intersection conceptually as
\begin{equation}
\label{equation3.5}
W' = \{\ell \in W^N \mid \ell t \in s W^{N+1} + W^{d+g-1}\}.
\end{equation}
In practice, we will have $N = 2g-1+i$ with $i \geq 0$, so the above
computation measures the extent to which the sum
$sW^{2g+i} + tW^{2g-1+i} + W^{d+g-1}$ is not direct.  The reader
should compare this with equations \eqref{equation2.7}
and~\eqref{equation2.8}: there, the ``excess degree'' $i$
appeared in only one of the first two summands, and the sum was direct.

Analogously to Proposition~\ref{proposition3.1}, we begin our
discussion with a criterion to guarantee that the space $W'$ is
really equal to $W^N_A$.  This is the result that originally led us
to define typical divisors and to investigate their properties.

\begin{proposition}
\label{proposition3.5}
Let $D$ be given with $s \in W^{d+g}_D$ and $t \in W^{d+g+1}_D$.
Suppose that $\deg s = d+g$, $\deg t = d+g+1$, and $N \geq 2g-1$.
Assume further that while computing the space $W'$ of~\eqref{equation3.5},
we also compute $\hat{W} = W' \intersect W^{2g-1}$, and
determine that $\hat{W} = 0$.  Then $D$
was typical to begin with, and $W'$ is indeed equal to $W^N_A$, where
$A$ is the flip of $D$, and $\Divisor s = D+A$ as in our discussion.
Moreover, $A$ is semi-typical.

Conversely, if $D$ is typical, then $\hat{W}$ will equal zero,
and the above computation of $W^N_A$ is correct.  Thus the flip of
a typical divisor is always semi-typical.
\end{proposition}
\begin{proof}
Note first that $sW^{N+1} \intersect W^{d+g-1} = 0$, since a nonzero
multiple of $s$ must have degree at least $d+g$.  (Since
$\deg s = d+g$, we also see that the divisor $A$ has degree~$g$.)
Hence $\dim (sW^{N+1} + W^{d+g-1}) = (N+2-g) + d$.
If $\hat{W} = 0$, then
$tW^{2g-1} \intersect (sW^{N+1} + W^{d+g-1}) = 0$.  Thus 
$\dim (tW^{2g-1} + sW^{N+1} + W^{d+g-1}) = N+d+2 = \dim W^{N+d+g+1}$.
Since $N \geq 2g-1$, it follows that
$tW^{2g-1} + sW^{N+1} + W^{d+g-1} \subset W^{N+d+g+1}$, so we obtain
equality.  By the converse condition to~\eqref{equation2.8} in
Proposition~\ref{proposition2.7}, we obtain that $D$ is typical, as desired.
Hence $\{s,t\}$ is an IGS for $D$, 
and~\eqref{equation3.4} holds, so 
$W^N_A \subset W'$; in particular,
$W^{2g-1}_A \subset \hat{W} = 0$, so we deduce that $A$
is semi-typical.  It remains to show 
that $W' \subset W^N_A$.  Suppose that $\ell \in W'$ satisfies
$t \ell = s \ell' + \ell''$, with $\ell' \in W^{N+1}$ and
$\ell'' \in W^{d+g-1}$.  Then 
$\ell'' \in t W^N + s W^{N+1} \subset I_D$, so we conclude that
$\ell'' \in I_D \intersect W^{d+g-1} = W^{d+g-1}_D = 0$ by
semi-typicality.  Thus $t \ell = s\ell'$, so $\ell \in W^N_A$
by~\eqref{equation3.4}.  This proves the results in the first
paragraph.

The converse holds because, when $D$ is typical, 
if any nonzero element $\ell \in W^{2g-1}$
satisfying~\eqref{equation3.5} existed, it would give rise to a
nontrivial linear dependence 
between the subspaces $tW^{2g-1}$, $sW^{N+1}$, and $W^{d+g-1}$,
contradicting the direct sum decomposition in~\eqref{equation2.8}.
\end{proof}

\begin{corollary}
\label{corollary3.5.5}
If $D$ is a typical divisor of degree $d=g$, then $D$ is reduced in
the sense of Proposition~\ref{proposition3.4.5}.
\end{corollary}
\begin{proof}
The converse in Proposition~\ref{proposition3.5} tells us that
the flip $A$ of $D$ is semi-typical.  Since $s \in W^{2g}_{D+A}$, we
conclude that $D$ is also the flip of $A$, so by applying
Proposition~\ref{proposition3.4.5} to~$A$, we deduce that $D$ is
reduced.
\end{proof}

The following is the algorithm that corresponds to
Proposition~\ref{proposition3.5}.  We state it for general
$d$, but in fact will apply it mainly when $d=g$ or $d=2g$.

\begin{proposition}
\label{proposition3.6}
Make the same assumptions on $D$, $s$, $t$, and $N$ as in
Proposition~\ref{proposition3.5}.
The following algorithm succeeds if and only if $D$ is typical, and,
upon success, correctly computes $W^N_A$, and certifies that the input
$D$ was typical and that the output $A$ is semi-typical.

\smallskip 
\noindent  
\emph{Algorithm:} 
\begin{enumerate}
\item 
Compute an $(N+2)\times (N-g+2)$ matrix $M'$ (analogous to the last six
columns of the matrix $N'$ in Section~9 of~\cite{FASKKM}), whose columns
describe an echelon basis for the image of 
$sW^{N+1}$ in $W^{N+d+g+1}/W^{d+g-1}$.  (This amounts to multiplying $s$ by
each ``monomial'' in $W^{N+1}$ in order of increasing degree, and
ignoring the $d$ terms of ``lowest degree'' in each result.)  It
follows that reducing modulo the columns of $M'$ implements the
reduction map from the $(N+2)$-dimensional space
$W^{N+d+g+1}/W^{d+g-1}$ to the $g$-dimensional space
$V = W^{N+d+g+1}/(W^{d+g-1}+sW^{N+1})$.
\item
Take a similar echelon basis for the image of $tW^N$ in
$W^{N+d+g+1}/W^{d+g-1}$, and use the matrix $M'$ to reduce each element of
this basis into $V$.  Make a new $(N-g+1)\times g$ matrix $M''$ whose
columns are the reductions of these basis elements.  Thus the leftmost $g$
columns of $M''$ represent the reductions of $tW^{2g-1}$ to the space $V$.
\item
Perform Gaussian elimination on $M''$ to find its kernel, which
corresponds to the space $W'$ of~\eqref{equation3.5}.  Along the way,
compute $\hat{W}$ as the kernel of the leftmost $g\times g$ submatrix
of $M''$.  If $\hat{W} \neq 0$, then return ``fail''.
\item
Otherwise, compute an echelon basis for the kernel of $M''$; analogously to
Proposition~\ref{proposition3.4}, this produces $N-2g+1$ elements
$\ell \in W^N$, ordered by degree,
that satisfy
equation~\eqref{equation3.4}.  Return ``succeed'', along with these
elements as a basis for $W^N_A$.
\end{enumerate}
\end{proposition}

\begin{example}
\label{example3.6.5}
We illustrate the above algorithm on the results of Section~9
of~\cite{FASKKM}.  In that context, the divisor that we wish to flip
is written $D+D'$, of degree $d=2g=6$, and we know that $D+D'$ is
semi-typical (this follows from Proposition~\ref{proposition3.3}
above, when $D \neq D'$, and from Proposition~\ref{proposition3.11} below,
when $D=D'$).  In particular, we know
elements $s \in W^9_{D+D'}$ and $t \in W^{10}_{D+D'}$.  In flipping
this divisor, we have $\Divisor s = D + D' + D''$ with $\deg D'' = g = 3$,
and we wish to compute $W^N_{D''}$ for $N = 7$.  We calculate $W'$ as
in~\eqref{equation3.5} by working in the quotient space
$V = W^{N+d+g+1}/(sW^{N+1} \directsum W^{d+g-1})
   = W^{17}/(sW^8 \directsum  W^8)
   = W^{17}/(sW^8 + W^9)$,
where the last equality follows from $\fieldK s + W^8 = W^9$.  We have
$\dim V = 3$, and we construct in that article a matrix $M''$ whose
columns represent the images in $V$ of the basis
$\{t,xt,yt,x^2 t, xyt\}$ for $tW^7$.  Thus an element of $\ker M''$
describes a linear combination of $c_0 t + c_1 xt + \dots,$ that lies in
$tW^7 \intersect (sW^8 \directsum W^8)$, hence simultaneously describes a
combination $\ell = c_0 + c_1 x + \dots \in W'$.  Now in Proposition~9.3
of~\cite{FASKKM}, we simultaneously invert two field elements
$\beta_2$ and $\gamma_4$; in that context, invertibility of $\beta_2$
means that the first three columns of $M''$ are linearly independent.
These columns represent the images of $t, tx, ty$ in $V$, where we
recall that $\linalgspan\{1,x,y\} = W^4 = W^5$, so independence of
these three columns means that $W' \intersect W^5 = 0$.  Hence by
Proposition~\ref{proposition3.5}, we deduce (provided we are able to
invert $\beta_2$) that $D+D'$ is indeed typical, and that
$W' = W^7_{D''}$, as desired.  Our results show that $D''$ is
semi-typical, but in this case we can prove the stronger result that
$D''$ is typical, provided $\gamma_4$ is also invertible.  To see
this, refer to the last paragraph of Section~9 of~\cite{FASKKM}, where
one sees that calculating $\ker M''$ produces elements of the form
$F'' = x^2 + a''y + b''x + c'', G'' = xy + d''y + e''x + f''
\in W^7_{D''}$.  Here, somewhat miraculously, $a'' = -\gamma_4$, so it
is invertible, and hence Proposition~\ref{proposition2.9} tells us
that $D''$ is typical, which goes beyond our result for arbitrary $C$.
\end{example}

We have thus shown:

\begin{proposition}
\label{proposition3.7}
In the $C_{3,4}$ case, suppose that $s \in W^9_{D+D'}$ and
$t\in W^{10}_{D+D'}$ are as given in the input of Sections 8 and~9
of~\cite{FASKKM}.  If calculations of those sections can be
carried out, including the inversion of the product
$\beta_2 \gamma_4$, then the final result of that calculation
correctly computes the ``flip'' $D''$ of $D+D'$, and it also certifies
that $D''$ is typical.  Conversely, if $D+D'$ and $D''$ are
both\footnote{In fact, $D+D'$ is typical if and only if $D''$ is, due
to Corollary~\ref{corollary3.8.5}.
Indeed, $D+D'$ and $D''$ are flips of each other, up to the
equivalence of Corollary~\ref{corollary2.8}.}
typical, then the product $\beta_2 \gamma_4$ can be successfully
inverted, and the calculation succeeds.
\end{proposition}

\begin{example}
\label{example3.7.5}
We now apply Proposition~\ref{proposition3.5} to the case of
flipping a divisor of degree~$d=3$ on a $C_{3,4}$ curve.  This is
needed in Subsection~2.3 and Section~10 of~\cite{FASKKM}.  Consider a
typical divisor $D$, described as usual in terms of $\{F,G\}$ instead
of $\{s,t\}$, where $F = x^2+ay+bx+c \in W^6_D$ (with $a\neq 0$) and 
$G = xy+dy+ex+f \in W^7_D$.  We can write
$\Divisor F = D + A, \Divisor G = D + B$
where this time we know that $A$ and $B$ are disjoint,
and~\eqref{equation3.4} holds.  In our computation of the 
``flip'' of $D$, we obviously know that $F \in W^6_A$, so our goal
is to find an element $G_1 \in W^7_A$, where $A$ is described
by $\{F, G_1\}$; hence we wish to apply
Proposition~\ref{proposition3.5} with $N=7$.  The desired element
$G_1$ must satisfy $G_1G \in FW^8$, and as in the proof of our
proposition, it is enough to know that $G_1G \in FW^8 + W^5$, since
any possible difference between $G_1 G$ and an element of $FW^8$ must
belong to $(F\Raff + G\Raff) \intersect W^5 = W^5_D = 0$.
Now in Equation~(19) of~\cite{FASKKM}, we exhibit specific elements
$G_1 \in W^7, H_1 \in W^8$ that satisfy $GG_1 + FH_1 \in W^4 = W^5$,
so this $G_1$ is our desired element.  (In that equation, we wrote
$G''', H$ instead of $G_1,H_1$, but we do not want to cause confusion
with our notation $H$ from this article; besides, the
notation $G_1, H_1$ appears in Section~5 of~\cite{FASKKM} with the same
meaning that we wish to use now.)
\end{example}

We deduce the following result.


\begin{proposition}
\label{proposition3.8}
In the $C_{3,4}$ case, let $D$ be a typical divisor of degree~$3$,
described by $F,G$ as above.  Compute $G_1 \in W^7, H_1 \in W^8$ such
that $GG_1 + FH_1 \in W^5$, as described in the above paragraph.  Then
(i) $GG_1 + FH_1 = 0$; (ii) $\{F,G_1\}$ are an IGS for the complementary
divisor $A$ of $D$; (iii) $A$ is typical; (iv) the divisors of $G_1$
and $H_1$ have the form $\Divisor G_1 = A + E$ and $\Divisor H_1 = B + E$;
and (v) the divisors $A$ and $B$ are disjoint, the divisors $D$ and
$E$ are disjoint, $\deg A = \deg D = 3$, and $\deg B = \deg E = 4$.
\end{proposition}
\begin{proof}
By the discussion preceding the proposition, we have $F \in W^6_A$ and
$G_1 \in W^7_A$, and (i) holds.  Moreover, $F = x^2 + ay + bx + c$
with $a \neq 0$, so we deduce from Proposition~\ref{proposition2.9}
that $A$ is typical.  This yields (ii) and (iii).  The divisor of
$G_1 \in W^7_A$ must have the form $A+E$, and 
$\{F,G_1\}$ are an IGS for $A$, so $E$ is disjoint from $D$.  Finally,
the divisor of $H_1$ follows from the fact that $FH_1 = -GG_1$ has
divisor $\Divisor G + \Divisor G_1 = D + B + A + E$.  This shows (iv)
and (v), and completes the proof.
\end{proof}

\begin{corollary}
\label{corollary3.8.5}
In the $C_{3,4}$ case, a
divisor $D$ is typical if and only if its ``flip'' $A$ is typical.
\end{corollary}
\begin{proof}
We have just seen this in case $d=3$, since both $D$
and $A$ share the same $F\in W^6$ with $a \neq 0$; this uses
Part~(2) of Proposition~\ref{proposition2.9}.  For a higher
degree divisor, the result follows from Corollary~\ref{corollary2.8},
since $D - d\Pinf$ is equivalent to a divisor of the form
$D' - 3\Pinf$, and $D$ and $D'$ will have the same ``flip''.
\end{proof}

For arbitrary $C$, we suspect that typicality is not preserved by
flipping.  However, in the situation generalizing
Proposition~\ref{proposition3.8}, we are likely in practice to
encounter a divisor $D$ with $\deg D=g$ as the result of flipping a 
previous divisor $\tilde{D}$; see for example Sections 10 and~11
of~\cite{FASKKM}, where our current triple $(\tilde{D},D,A)$
corresponds to $(D+D',D'',D''')$ in that article.

\begin{proposition}
\label{proposition3.9}
Let $C$ be arbitrary.  Suppose that $D$ is a semi-typical divisor with
$d = g$, and take as usual $F \in W^{2g}_D, G \in W^{2g+1}_D$ with 
$\Divisor F = D+A$, $\Divisor G = D+B$.
Assume that $D$ was originally obtained as
a successful flip of a divisor $\tilde{D}$, using the algorithm in
Proposition~\ref{proposition3.6}.  Suppose we now use the algorithm a
second time, and it successfully computes
$W^N_A$ for some $N \geq 2g+1$.  Then $A$ and $D$ are both
typical, and the echelon basis for $W^N_A$ computed by
our second application of the algorithm begins with the same element
$F \in W^{2g}_A$, and a new element $G_1 \in W^{2g+1}_A$.  As a
byproduct of this second application, based on~\eqref{equation3.5}
(and using $(F,G)$ for $(s,t)$), we also obtain an element $H_1 \in
W^{2g+2}$ for which $GG_1 + FH_1 \in W^{2g-1}$.  Then the conclusions
of~Proposition~\ref{proposition3.8} hold, with the slight
modification that $\deg A = \deg D = g$ and $\deg B = \deg E = g+1$.

Conversely, if $D$ is obtained as the flip of $\tilde{D}$ as above,
and $D$ is typical, then the second flip that computes $A$ will
succeed using the algorithm in Proposition~\ref{proposition3.6}.
\end{proposition}
\begin{proof}
Only the first collection of statements needs proof; the converse is
included in Proposition~\ref{proposition3.6}.

Upon successful completion of the computation, the
divisors $\tilde{D}$ and $D$ are certified to be typical, because
both the first and second uses of the algorithm are certified by
Proposition~\ref{proposition3.6}.  Thus $A$ is also 
typical by Corollary~\ref{corollary2.8}, because $A - g\Pinf$ is
linearly equivalent to $\tilde{D} - \tilde{d}\Pinf$, both being
linearly equivalent to $g\Pinf - D$ (the ``negation'' of the class
$[D-g\Pinf]$).  We have $GG_1 + FH_1 \in W^{2g-1}$
from~\eqref{equation3.5}, and as usual $GG_1 + FH_1 \in G\Raff +
F\Raff = I_D$, so in fact $GG_1 + FH_1 \in W^{2g-1}_D = 0$.  The rest
of the proof is a similarly direct adaptation of the proof of
Proposition~\ref{proposition3.8}.
\end{proof}

\subsection{Doubling a divisor}
\label{subsection3.3}

Our goal in this subsection is to
compute the space $W^N_{2D}$, for suitable~$N$, when given a divisor
$D$.  For convenience, we will restrict to semi-typical $D$ of degree $d=g$.
Our computation of $W^N_{D+D'}$ via an intersection in
Subsection~\ref{subsection3.1} cannot be used directly with $D'=D$.
Instead of looking at elements of $W^N_{D'}$ which vanish at $D$, we
can proceed as in Section~5 of~\cite{FASKKM},
where we set up a system of equations for $W^N_{2D}$ to compute
sections $\ell \in I_D$ whose
differential $d\ell$ also vanishes at $D$.  In this article, we
set up exactly the same system of equations, but justify
correctness of the equations from two new perspectives.  We
believe that both the old and the two new points of view have value,
and we encourage the reader to compare the treatment
here with the one in~\cite{FASKKM}.  The following is the system of
equations and the analogous algorithm to our previous article.

\begin{proposition}
\label{proposition3.10}
Suppose that $D$ is semi-typical of degree $g$.  Let $D$ be
described as usual by $F \in W^{2g}_D$, $G \in W^{2g+1}_D$ (with
$\deg F = 2g$, $\deg G = 2g+1$), and write
$\Divisor F = D + A$, $\Divisor G = D + B$.  Now suppose, similarly to
Proposition~\ref{proposition3.9}, that we successfully use the
flipping algorithm of Proposition~\ref{proposition3.6} to compute the basis
$\{F, G_1\}$ of $W^{2g+1}_A$, alongside $H_1 \in W^{2g+2}$ for which
$GG_1 + FH_1 = 0$.  As before, write $\Divisor G_1 = A+E$,
$\Divisor H_1 = B+E$.  At this point, the success of the flipping
algorithm guarantees that $A$ is semi-typical and $D$ is typical, so
in particular $A$ and $B$ are disjoint; however, unlike
Proposition~\ref{proposition3.9}, we do not assume that we obtained
$D$ as the successful flip of some $\tilde{D}$, so possibly $A$ might
not be typical, and we cannot assert (yet) that $D$ and $E$ are disjoint.

Let $N \geq 3g-1$, and compute the space $W''$ defined by
\begin{equation}
\label{equation3.6}
W'' = \{\ell = aF + bG \in W^N_D \mid a,b \in \Raff \text{ with }
\ell' := aG_1 - bH_1 \in W_D^{N+1}\}.
\end{equation}
Moreover, suppose that our computation also yields
$\hat{W} := W'' \intersect W^{3g-1}$.
If $\hat{W} = 0$, then $A$ is also typical, $W''$ correctly computes
$W^N_{2D}$, and $2D$ is semi-typical.

Conversely, if $A$ is typical and $2D$ is semi-typical, then $\hat{W}
= 0$, and the above algorithm succeeds and correctly computes $W^N_{2D}$.
\end{proposition}
\begin{proof}
Before we begin, observe that $\ell' = (G_1/F)\ell$, because $GG_1 =
-FH_1$; also observe that the full divisor (i.e., including $\Pinf$) of
$\ell'/\ell = G_1/F$ is 
$\Divisor G_1 - \Divisor F = (A+E-(2g+1)\Pinf) - (A+D - 2g\Pinf)
  = E - D - \Pinf$.
In particular, $\deg \ell' = \deg \ell + 1$ with $\deg \ell \leq N$,
so $\ell' \in W^{N+1}$ automatically; the significant condition on
$\ell'$ is that it should belong to $I_D$.

We give two proofs of our result.  For the first proof, we know that
$F\Raff = I_{D+A}$ and $F\Raff + G\Raff = I_D$.  As $D$ and $E$ are
not necessarily disjoint, let $A' = \gcd(D,E)$, and write
$D = A' + D'$, $E = A' + E'$, with $D',E'$ disjoint.  Then
$F\Raff + G_1\Raff = I_{A+A'}$.
Now an element $\ell$ 
belongs to $W''$ if and only if it satisfies the following conditions:
(i) $\ell \in W^N$, or equivalently $\ell' \in W^{N+1}$;
(ii) $\ell \in I_D$; and 
(iii) $\ell G_1 =  F\ell'$ belongs to $FI_D = I_{2D+A}$.  
Note however that condition (ii) is equivalent to
having $\ell F \in FI_D = I_{2D+A}$.  Thus we see that conditions (ii)
and~(iii) mean that $\ell$ belongs to a certain colon ideal, namely
$\ell \in (I_{2D+A} : F\Raff + G_1\Raff) = (I_{2D+A} : I_{A+A'}) = I_{D+D'}$.
Thus $W'' = W^N_{D+D'}$, and $\hat{W} = W^{3g-1}_{D+D'}$.  Now the
fact that $\hat{W} = 0$ forces $\deg (D+D') \geq 2g$, which means that
necessarily $D=D'$, $A'=0$, and $D,E$ are disjoint.  Thus we have
successfully computed $W^N_{2D}$, with $W^{3g-1}_{2D} = 0$, i.e., $2D$
is semi-typical.  Finally, to see that $A$ is typical, we argue as in
Proposition~\ref{proposition2.8.4} that the inclusion $FW^{2g} +
G_1W^{2g-1} \subset W^{4g}_A$ must be an equality.  As usual, it is
enough by dimension-counting to check that
$FW^{2g} \intersect G_1W^{2g-1} = 0$.  But this intersection is
precisely $W^{4g}_{A+D} \intersect W^{4g}_{A+E} = W^{4g}_{A+D+E}$,
because we know that $D$ and $E$ are disjoint.  Moreover,
this last space is isomorphic (via division by $G_1$) to
$W^{2g-1}_D = 0$.  After this, we again argue that
$W^{4g}_A + W^{2g-1} = W^{4g}$ since $A$ is semi-typical.

Conversely, suppose we know from the start that $A$ is typical.  Then
$D,E$ are necessarily disjoint, and $F\Raff + G_1\Raff = I_{A}$, so we
immediately obtain $W'' =  W^N_{2D}$ and $\hat{W} = W^{3g-1}_{2D} = 0$
by semi-typicality of $2D$.  Our first proof is now complete.

Our second proof, which we generalize below, is to consider the
definition of $W''$ as a system of equations for $\ell'$ instead of for
$\ell$.  From our knowledge of the divisor of $\ell'/\ell = G_1/F$, we see
that the condition $\ell \in W^N_D$ corresponds to $\ell' \in W^{N+1}_E$;
hence our calculation is equivalent to 
computing $W''' = \{\ell' \in W^{N+1}_E \mid \ell' \in W^{N+1}_D\} =
W^{N+1}_E \intersect W^{N+1}_D$.  This intersection is $W^{N+1}_{D+E}$
precisely when $D$ and $E$ are disjoint, which can be certified by the
condition $\hat{W} = 0$, since this condition is equivalent to $W'''
\intersect W^{3g} = 0$, which by Proposition~\ref{proposition3.1}
allows us to conclude that $D,E$ are disjoint and that $D+E$ is
semi-typical.  Now $[E] = [D+\Pinf]$, so $[D+E] = 
[2D+\Pinf]$, hence Corollary~\ref{corollary2.8} implies that
$2D$ is semi-typical if and only if $D+E$ is.  Finally,
$W'' = (F/G_1)W''' = (F/G_1)W^{N+1}_{D+E} = W^N_{2D}$, as desired.
The remaining assertions follow similarly to the first proof above.
So, in essence, we use $G_1/F$ to move between the 
class of $D-g\Pinf$ and the equivalent class of $E-(g+1)\Pinf$, and we
replace doubling $D$ with adding $D+E$, which can be carried out using our
earlier methods.  We then move back within the equivalence class to $2D$.
\end{proof}

We make some remarks on how one computes the space $W''$ in practice.
Since $D$ is typical, we have a direct sum
$W^{2g+i}F + W^{2g-1}G = W^{4g+i}_D$ for $i \geq 0$.
This allows us to proceed smoothly if $N = 4g+i \geq 4g$, by taking all
$a \in W^{2g+i}$ and $b \in W^{2g-1}$, but in practice we
want $N = 3g+1$.  In that case, in setting up a system of equations
for $\ell = aF + bG$, we must restrict the possible values of pairs
$(a,b) \in W^{2g} \times W^{2g-1}$ to ensure that $\deg \ell \leq N$.

\begin{example}
\label{example3.10.5}
In the $C_{3,4}$ case, we saw in 
Proposition~\ref{proposition2.9} that $W^{12}_D = FW^6 + GW^5$ has a basis
$\{F,G, yF-xG, xF, xG, yG, x^2F\}$, ordered by degree, of which the
first five elements are a basis for $W^{10}_D$.  Hence the pairs
$(a,b)$ to consider are $\fieldK$-linear combinations of 
$\{(1,0),(0,1),(y,-x),(x,0),(0,x)\}$.  In Sections 5 and~6 of~\cite{FASKKM},
when $D=D'$, we looked for elements $\ell \in W''$ of the form
$\ell = c'_1 F + c'_2 G + c'_3(yF-xG) + c'_4 xF + c'_5 xG$.  Now the
corresponding $\ell'$ is
$\ell' =  c'_1 G_1 - c'_2 H_1 + c'_3(yG_1+xH_1) + c'_4 xG_1 - c'_5 xH_1$,
and we want to set up a system of equations that ensures that
$\ell'$ has zero image in the three-dimensional quotient $W^{11}/W^{11}_D$.
This is exactly the kernel of the matrix $M'$ in Section~6
of~\cite{FASKKM}.  Moreover,
the first three columns of $M'$ correspond to taking $\ell$ to be a linear
combination of $F, G, yF-xG$, or respectively to $\ell'$ being a linear
combination of $G_1, -H_1, yG_1 + xH_1$, and invertibility of the leftmost
$3\times 3$ submatrix of $M'$ is exactly the condition that
$0 = W'' \intersect \linalgspan(F, G, yF - xG) = W'' \intersect W^8_D$, or
respectively that $0 = W''' \intersect W^9_E$.  This is exactly what we
need to apply Proposition~\ref{proposition3.10}.
\end{example}
Combining the above with Proposition~\ref{proposition3.8} to compute
first $\{G_1, H_1\}$, we obtain:
\begin{proposition}
\label{proposition3.11}
In the $C_{3,4}$ case, let $D$ be a typical degree $3$ divisor.  If the
computations in Sections 5, 6, and~7 of~\cite{FASKKM} can be carried out,
including the inversion in Proposition~6.1 of that article, then the
computation succeeds, and correctly returns $s \in W^9_{2D}, t \in
W^{10}_{2D}$, if and only if $2D$ is semi-typical.  In all cases, the
kernel of $M'$ will compute the space $W^{10}_{2D}$.
\end{proposition}
Generalizing this method from the $C_{3,4}$ case to an arbitrary curve
$C$ is straightforward, once one takes into account some possibly more
complicated conditions on the pairs $(a,b) \in W^{2g} \times W^{2g-1}$
that one wishes to consider in the system of equations.  We will leave
the details of a general algorithm in that case to the reader.

We now turn to our second method for doubling.  This time, we begin
with a divisor $\tilde{D}$ whose flip will be the divisor $D$ that we
wish to double.  This is analogous to combining
Proposition~\ref{proposition3.9} with the second proof in
Proposition~\ref{proposition3.10}.  The idea is to combine
a slight extension of the flipping algorithm that produces
$D$ with the ideas of Proposition~\ref{proposition3.10}, in a way that
obtains both $W^{N+1}_D$ and $W^{N+1}_E$ for a 
suitable $N$; here $E$ is the same as in
Proposition~\ref{proposition3.10}, using $\tilde{D}$ instead of $A$.
The intersection $W^{N+1}_D \intersect W^{N+1}_E = W^{N+1}_{D+E}$ can
then be transferred back to give $W^N_{2D}$.

\begin{proposition}
\label{proposition3.12}
Let $\tilde{D}$ be a semi-typical divisor of degree $\tilde{d}$, and
assume given $\tilde{s} \in W^{\tilde{d}+g}_{\tilde{D}}$,
$\tilde{t} \in W^{\tilde{d}+g+1}_{\tilde{D}}$ as usual.  Write
$\Divisor \tilde{s}  = D + \tilde{D}$ and
$\Divisor \tilde{t} = E + \tilde{D}$,
with $\deg D = g$ and $\deg E = g+1$.
For $N \geq 3g-1$, the following algorithm, if 
successful, simultaneously computes, for both the flipped divisor $D$
and its double $2D$, the spaces
$W^{N+1}_D$ and $W^{N}_{2D}$.  The algorithm succeeds in computing
$W^{N+1}_D$ and $W^{N}_{2D}$ if and only if $\tilde{D}$ is typical,
and correctly certifies semi-typicality of $2D$ whenever it holds.

\smallskip 
\noindent  
\emph{Algorithm:} 
\begin{enumerate}
\item
Set up a system of equations that computes the space of pairs
\begin{equation}
\label{equation3.7}
\tilde{W} = 
\{(\ell,\ell') \in W^{N+1} \times W^{N+2} \mid
   \tilde{t} \ell + \tilde{s} \ell' \in W^{\tilde{d}+g-1} \}.
\end{equation}
Concretely, make an $(N+3) \times (2N - 2g + 5)$ matrix $\tilde{M}$
concatenating matrices similar to those produced by Steps 1
and~2 of the algorithm in Proposition~\ref{proposition3.6}.  More
precisely, the columns
of $\tilde{M}$ consist of $N-g+3$ columns representing the images
(in $W^{N+\tilde{d}+g+2}/W^{\tilde{d}+g-1}$) of a basis for
$\tilde{s} W^{N+2}$, and of $N-g+2$ other columns representing the
images  of a basis for $\tilde{t} W^{N+1}$.  The matrix should be
computed in terms of products $\tilde{s}m$ (respectively,
$\tilde{t}m$) over various ``monomials'' $m$, indexed by degree, that
give a basis for $W^{N+2}$ (respectively, $W^{N+1}$), and the products
should themselves also be expressed in terms of a basis of monomials
describing $W^{N+\tilde{d}+g+2}/W^{\tilde{d}+g-1}$.
This is analogous to the entire matrix $N'$ in Section~8 of~\cite{FASKKM}.
\item
Find $\ker \tilde{M}$ (i.e., the space $\tilde{W}$) in a way that
simultaneously verifies that a 
certain $(2g+1)\times(2g+1)$ submatrix of $\tilde{M}$ is invertible.
Specifically, detect whether the columns describing
$\{ \tilde{s}m \mid \deg m \leq 2g \} \union
\{ \tilde{t}m \mid \deg m \leq 2g-1 \}$
are linearly independent; note that these columns all correspond to
elements of
$W^{\tilde{d}+3g}/W^{\tilde{d}+g-1}
\subset W^{N+\tilde{d}+g+2}/W^{\tilde{d}+g-1}$,
so only the top $2g+1$ entries are nonzero.  In practice, these
columns should be placed as the leftmost columns of~$\tilde{M}$, and
one computes an echelon basis for $\ker \tilde{M}$.
If these columns are not linearly independent, then return ``fail''.
\item
Success at the previous step certifies that $\tilde{D}$ is typical, hence
that $D$ (and also $E$) is semi-typical; every pair
$(\ell,\ell') \in \tilde{W}$ actually satisfies $\tilde{t} \ell +
\tilde{s} \ell' = 0$.  With respect to a suitable ordering by degree,
one can find an echelon basis for $\tilde{W}$ of the form 
$\{(\ell_j,\ell'_j) \mid 1 \leq j \leq N - 2g + 2\}$, with
$\deg \ell_j = 2g-1+j$, $\deg \ell'_j = 2g - j$.  Moreover,
$\{\ell_1, \dots, \ell_{N-2g+2}\}$ is a basis for $W^{N+1}_D$, while
$\{\ell'_1, \dots, \ell'_{N-2g+2}\}$ is a basis for $W^{N+2}_E$.
Discarding $\ell'_{N-2g+2}$, we actually have that 
$\{\ell'_1, \dots, \ell'_{N-2g+1}\}$ is a basis for $W^{N+1}_E$.
\item
Compute $W^{N+1}_{D+E}$ as the intersection
$W^{N+1}_D \intersect W^{N+1}_E$, in a way that also identifies
$\hat{W} = W^{3g}_{D+E}$ (this is easy to do with echelon bases).
Then $\hat{W} = 0$ if and only if $D+E$, equivalently $2D$, is
semi-typical.
In practice, one should compute the intersection by computing
the space of tuples
$(c_1,\dots, c_{N-2g+1})$ which satisfy
$\sum_j c_j \ell'_j = 0 \in W^{N+1}/W^{N+1}_D$.
\item
For each tuple $(c_j)$ as above (take a basis of the space of such
tuples), compute $\sum_j c_j \ell_j \in W^{N}_{2D}$.  The collection
of such $\sum_j c_j \ell_j$ gives a basis for $W^{N}_{2D}$, which can
be arranged to be in echelon form due to our control of the degrees
$\deg \ell_j$.
\end{enumerate}
\end{proposition}

\begin{proof}
If our computation passes Step~(2), we conclude that
$\tilde{s} W^{2g} + \tilde{t} W^{2g-1}$ generates all of
$W^{\tilde{d}+3g}/W^{\tilde{d}+g-1}$, from which it follows that
$\tilde{D}$ is typical and that $D,E$ are disjoint.  Moreover, $D$ is
semi-typical, being the flip of the typical divisor $\tilde{D}$, and
$E$ is also typical by Corollary~\ref{corollary2.8} (using
the fact that the full divisor $\Divisor(\tilde{t}/\tilde{s})$ is
$E-D-\Pinf$).  Moreover, all combinations
$\tilde{t}\ell + \tilde{s}\ell'$ belong to $I_{\tilde{D}}$, so a pair
$(\ell,\ell') \in \tilde{W}$ must satisfy
$\tilde{t}\ell + \tilde{s}\ell' \in W^{\tilde{d}+g-1}_{\tilde{D}} = 0$.
This proves the first assertions of~(3).

The next assertions of~(3) boil down to observing that projecting from
$(\ell,\ell') \in \tilde{W}$ to the $\ell$ component is
equivalent to our usual algorithm for flipping $\tilde{D}$  to find
$W^{N+1}_D$, as in~\eqref{equation3.5}; similarly for having the
$\ell'$ compute $W^{N+2}_E$.  Another way to see this last fact is
that $\tilde{W}$ is the graph of the bijection
$\ell \mapsto \ell' = -(\tilde{t}/\tilde{s})\ell$ between $W^{N+1}_D$
and $W^{N+2}_E$.  By semi-typicality of either of these spaces, there
exist corresponding ``triangular'' bases of elements $\ell_j$ or
$\ell'_j$ with the degrees that we claim.  Hence (3)~is now proved.

The intersection in~(4) is $W^{N+1}_{D+E}$ as claimed, since we have
already observed that $D$ and $E$ are disjoint; the comment on
semi-typicality of $D+E$ is immediate, as is the fact that $2D$ will
then also be semi-typical by Corollary~\ref{corollary2.8}.  Finally, the
correspondence in~(5) between
$\sum_j c_j \ell'_j \in W^{N+1}_E \intersect W^{N+1}_D
  = W^{N+1}_{D+E}$
and the corresponding $\sum_j c_j \ell_j$ is precisely multiplication
by $-(\tilde{s}/\tilde{t})$, which transforms the space
$W^{N+1}_{D+E}$ into the space $W^N_{2D}$.
\end{proof}

\section{Jacobian arithmetic for typical divisor classes}
\label{section4}

In this section, we assemble the results from the previous section,
with specific choices of parameters, to give algorithms for typical
divisor classes that work for the Jacobian of an arbitrary curve $C$
with a rational point $\Pinf$.

Before doing so, we collect here the final statement of our results for the
special case of $C_{3,4}$ curves and our previous algorithms:

\begin{theorem}
\label{theorem4.1}
The algorithms of~\cite{FASKKM} for addition and doubling in the Jacobian
of a $C_{3,4}$ curve work correctly with typical
divisors as input, and yield typical divisors as output, if and only
if the two $\fieldK$-inversions in each algorithm can be carried out. 
\end{theorem}
\begin{proof}
All divisors $D$ of degree $3$ are represented by elements $F,G$ as in
Proposition~\ref{proposition2.9}, by storing the elements $a,b,\dots, f$ as
well as the inverse $a^{-1}$.  Thus $a \neq 0$, and $D$ is typical.
Moreover, the algorithm for the ``addflip'' of two divisors ($D,D' \mapsto
D''$ with $\Divisor s = D + D' + D''$; see below) produces the correct
answer, with $D''$ typical due to $a'' \neq 0$ by Propositions
\ref{proposition3.3} and~\ref{proposition3.7}.  A similar result holds
when we compute the addflip for $D=D'$ (called ``doubleflip'' below),
by using Proposition~\ref{proposition3.11} instead of
Proposition~\ref{proposition3.3}.  The final inversion to obtain the
sum or double in the Jacobian (Section~10 of~\cite{FASKKM}) is correct
by Proposition~\ref{proposition3.8}.
\end{proof}

In the appendix, we give formulas for a speedup of the algorithms
of~\cite{FASKKM} by approximately 15\%, arising from revisiting the
previous work in light of the considerations that led to
Proposition~\ref{proposition3.12}.

We now address the generalization to arbitrary $C$.  Following the
discussion at the beginning of Section~\ref{section3}, we represent
all typical elements of $\Pic^0(C)$ as $[D - g\Pinf]$
for a unique typical divisor $D$ of degree~$g$.  We will occasionally
relax this to assume merely that $D$ is semi-typical.

As we have already discussed, ``flipping'' $D$
corresponds to negation in the Jacobian.  The other basic operation in
the Jacobian is the ``addflip'' operation, in the terminology we
introduced in~\cite{KKM}.  For the rest of this section, it is
convenient for us to separate this operation into two cases:
\begin{enumerate}
\item
The first case, which we continue to call ``addflip'', takes as
input two typical divisors $D,D'$, that one typically hopes are
disjoint.  Then we want to produce an output divisor $D''$ for which
there exists a degree $3g$ element $s \in W^{3g}_{D+D'}$ with divisor
$\Divisor s = D + D' + D''$.  This means that
$[D'' - g\Pinf] = -([D - g\Pinf] + [D' - g\Pinf])$.
\item
The second case, which we call ``doubleflip'', is the analog of
the above in the situation where $D=D'$.  Thus
$[D'' - g\Pinf] = -2[D-g\Pinf]$ in $\Pic^0(C)$, and
$s \in W^{3g}_{2D}$ satisfies $\Divisor s = 2D + D''$.
\end{enumerate}
The basic idea, of course, is that an addflip consists of carrying out
an addition, as in Subsection~\ref{subsection3.1} (producing a divisor
of degree $2g$), followed by a flip as in
Subsection~\ref{subsection3.2}, so the final answer is again a divisor
of degree~$g$.  A doubleflip is similar, except that the initial
addition is replaced by a doubling, as in
Subsection~\ref{subsection3.3}.  In doing this, we must take care to
specify (i) precise values of $N$, and (ii) choices of
algorithms to compute the spaces $W^N_E$ for various intermediate
divisors $E$; the goal in doing so is to ensure 
that we can always certify the final answer after an
addflip or a doubleflip to be typical.  Alternatively, we can
arrange to certify only that a subsequent flip of an addflip or a
doubleflip is itself typical, so that the basic operations on typical
classes in the Jacobian now become addition and doubling, with which
we must also include a certified way of doing negation.

In all the algorithms given below, if the algorithm fails, then it is
possible to use the general methods that work for all divisors
in~\cite{KKM,KKMAsymptotic}.  This will produce a constant slowdown,
since the linear algebra involves larger matrices, but is expected to
happen so very rarely in typical applications that it can be ignored.

\subsection{Direct generalization of the algorithms in~\cite{FASKKM}}
\label{subsection4.1}

We represent our typical degree $g$ divisor $D$ by a triple
$(F,G,G_1)$.  Here $D$ is determined by the pair $(F,G)$, with $F \in
W^{2g}_D, G \in W^{2g+1}_D$, and the ``flip'' $A$ of $D$ is a
degree~$g$ divisor, described by the pair $(F,G_1)$.  This
uses the same $F \in W^{2g}_A$ (because $\Divisor F = D+A$, as usual),
and $G_1 \in  W^{2g+1}_A$ is the same element as in
Proposition~\ref{proposition3.9}.  We assume that for any input to our
algorithms, representing such a divisor $D$ and its complement $A$, we
are guaranteed that $D$ and $A$ are both typical.

As written, there is some choice for $F$, $G$, and $G_1$.  For example,
$D$ determines the pair $(F,G)$ uniquely only up to replacing $(F,G)$
with $(aF,a'G+bF)$, where $a,a',b \in \fieldK$ with $a,a' \neq 0$, as
observed in Proposition~\ref{proposition2.5}. In terms of a basis of
``monomials'' for $\Raff$, we can make $F$ and $G$ unique by requiring
that (i) $F$ and $G$ are monic, and (ii) the coefficient in $G$ of the
monomial of degree $2g$ is zero.  As for the choice of $G_1$, note
that $A$ is uniquely determined by $D$ (or by $\Divisor F$), and $G_1$
can be modified in the same way as $G$, to yield a unique choice.

Putting together our previous results yields the following theorem.

\begin{theorem}
\label{theorem4.1.5}
In the above setting, we have the following algorithms for negation,
addition, and doubling of typical divisor classes, with necessary and
sufficient conditions under which the algorithms succeed.
\begin{enumerate}
\item
\textbf{Negating a divisor class:}
Replace the triple $(F,G,G_1)$ by $(F,G_1,G)$.
This exchanges the roles of $A$ and $D$, and succeeds in all cases.  Both
divisors remain typical.
\item
\textbf{Adding two (different) divisor classes:}
Briefly, carry out the algorithms in
Proposition~\ref{proposition3.4} then Proposition~\ref{proposition3.6}.
More specifically, let the two elements of the Jacobian come from divisors
$D$ and $D'$, represented by the triples $(F,G,G_1)$ and $(F',G',G'_1)$.
Carry out Step~(1) of Proposition~\ref{proposition3.4}, taking $N = 3g+1$.
In other words, compute an echelon basis of the space
$W^{4g}_D = FW^{2g} + GW^{2g-1}$, and select the first $g+2$ basis elements.
This yields a basis for the subspace $W^{3g+1}_D$; similarly,
compute $W^{3g+1}_{D'}$.  Now compute the intersection to obtain
$W^{3g+1}_{D+D'}$, as in the remaining steps of the algorithm in
Proposition~\ref{proposition3.4}.  This succeeds if and only if $D$ and
$D'$ are disjoint, and $D+D'$ is semi-typical, in which case 
we correctly obtain $s \in W^{3g}_{D+D'}, t \in W^{3g+1}_{D+D'}$.
Next, apply the flipping algorithm in Proposition~\ref{proposition3.6},
with $N= 2g+1$.  We write as usual $\Divisor s = D + D' + D''$.  The
flipping algorithm produces $\{F'',G''\}$ that describe $D''$; this
succeeds if and only if $D+D'$ is typical.  Do a further flip of $D''$ as in
Proposition~\ref{proposition3.6}, again with $N=2g+1$;
assuming this is successful, this produces a divisor
$D'''$, represented by elements $F''', G'''$, and certifies that $D''$
is typical.  We know that $D'''$ represents the same class in
$\Pic^0(C)$ as $D+D'$, so $D'''$ is typical by
Corollary~\ref{corollary2.8}.  Thus we return the triple
$(F''',G''',G'')$ (where in fact $F''' = F''$) as our representation
of the sum $[D'''-g\Pinf] = [D-g\Pinf] + [D'-g\Pinf]$.  Given that 
$D, D'$ and their flips are known to be typical before starting the
algorithm, this whole procedure succeeds if and 
only if $D,D'$ are disjoint, and $D''$ and $D'''$ are both typical.
\item
\textbf{Doubling a divisor class:}
Start with the usual input data $(F,G,G_1)$, and compute first 
$H_1 = -G G_1/F$, then $s,t \in W^{3g+1}_{2D}$ using
Proposition~\ref{proposition3.10} with $N= 3g+1$; thus $\Divisor s = 2D +
D''$.  Then do two flips, as in the algorithm for addition in~(2) above,
thereby obtaining representations of the divisors $D''$ and $D'''$ as in
that algorithm.  Success occurs if and only if $2D$ (hence also $D'''$) and
$D''$ are typical.
\end{enumerate}
\end{theorem}

\begin{remark}
\label{remark4.2}
Let us bound the probability of failure for the above two
algorithms for addition and doubling, when $\fieldK$ is a finite field
with $q$ elements, using the results we quoted from~\cite{KKMbrillnoether}
in Theorem~\ref{theorem2.8.6}.

For addition to fail, the inputs $D,D'$ and output
$D''$ must satisfy at least one of the following properties:
\begin{enumerate}
\item
The divisor $D$, or its flip, is not typical;
\item
The divisor $D'$, or its flip, is not typical;
\item
The divisor $D''$, or its flip $D'''$, is not typical;
\item
The divisors $D$ and $D'$ are not disjoint.
\end{enumerate}
The probability of at least one of these events happening is at most
the sum of their individual probabilities, which we will compute under
the uniform distribution on all triples of classes
$(x,x',x'') = ([D-g\Pinf],[D'-g\Pinf],[D''-g\Pinf]) \in (\Pic^0 C)^3$
with $x+x'+x'' = 0$.  Note first that the distribution of each of $x$,
$x'$, or $x''$ when looked at in isolation is uniform in $\Pic^0 C$.
In fact, any pair made from two of the three entries, such as
$(x,x')$, is uniformly 
distributed over $(\Pic^0 C)^2$, because this pair completely determines
the third entry (indeed, $x'' = -x-x'$).
It follows that each of events (1--3) above has probability at most
$(16^g \cdot 2g + 3.4)/q$.  Moreover, we claim that event (4) has a
probability at most $(1.7\cdot g)/q$.  This claim implies that the probability
of failure of addition is at most
$(3\cdot (16^g \cdot 2g + 3.4) + 1.7\cdot g)/q$.  For example, if
$g = 5$, then the numerator does not exceed $3.2\times 10^7$, which
means that if $q$ is, say, around $10^{30}$ (around $100$ bits), our
failure rate is below $3.2 \times 10^{-23}$.

We now explain why our claim
holds.  For this, it is enough to fix a good $D$ with $\deg D = g$,
and bound the probability that $D'$ is not disjoint from $D$.  Writing
$D = E_1 + \cdots + E_r$ as a sum of irreducible divisors with $\sum_j
\deg E_j = g$, the ``bad'' $D'$ are those of the form $E_j + E'$ with
$E'$ effective of degree $g - \deg E_j$.  Let $N_d$ denote the number
of effective degree $d$ divisors; then the number of bad $D'$ is at
most $\sum_j N_{g-\deg E_j}$.  The largest this can be is when the
$E_j$ are all distinct $\fieldK$-rational points, in which case our
upper bound for the number of bad $D'$ is $g N_{g-1}$, and the
probability of the pair $(D,D')$ not being disjoint is at most
$(g N_{g-1})/\abs{\Pic^0 C}$.  On the other hand, we know from
equation~(2.13) in Proposition~2.15 of~\cite{KKMbrillnoether} that 
$N_{g-1}/\abs{\Pic^0 C} \leq 1.7/q$, and this proves our claim.

A similar argument gives a bound for the probability that doubling will
fail.  Here the 
question is how often at least one of $D$ or $2D$ or their flips can
fail to be typical.  For uniformly random $D$, this is again at most
$(16^g \cdot 2g + 3.4)/q$.  However, the class of $2D$ is not uniform
in $\Pic^0 C$, unless $\abs{\Pic^0 C}$ happens to be odd, in which case
multiplication by~$2$ would be a bijection.  The worst-case scenario is
that the full $2$-torsion of the Jacobian is defined over $\fieldK$,
in which case multiplication by $2$ is a $2^{2g}$-to-$1$ map.  In that
case, the probability that $2D$ or its flip is not typical is at most
$2^{2g} (16^g \cdot 2g + 3.4)/q$.
(Indeed, if $B \subset \Pic^0 C$ is the bad set of elements which are
either not typical or whose flip is not typical, then the preimage of $B$
under multiplication by $2$ cannot have more than $2^{2g}\abs{B}$
elements.)  We deduce that the total probability of
failure is at most $(2^{2g} + 1)(16^g \cdot 2g + 3.4)/q$.  For our
sample parameters $g=5$ and $q \sim 10^{30}$, this probability is at
most $1.1\times 10^{-20}$.
\end{remark}

\subsection{A relative of the small model of~\cite{KKM}}
\label{subsection4.2}

We now describe a slight relaxation of the above algorithms, which is
close to the ``small model'' described in Section~5 of~\cite{KKM}.  In
this setup, we assume only that the divisor $D$ representing an
element of the Jacobian is semi-typical, but that its flip $\tilde{D}$
is typical; to guarantee correct results, some intermediate results in
our computations also need to be certified typical.  In this set of
algorithms, it occasionally helps to keep track of the penultimate
result found, in order to streamline a subsequent calculation.

In this setting, we represent a $D$ of degree $g$ by a triple
$(W^{3g+1}_D, \tilde{s}, \tilde{t})$.  The vector space $W^{3g+1}_D$
is described by an echelon basis, ordered as usual by degree, and we
assume that $D$ is semi-typical, and that $D$ is known to be equal to
the ``flip'' of a typical divisor $\tilde{D}$ of degree
$\tilde{d}$, with $\tilde{d} = g$ or~$2g$ (we also carry around the
value of $\tilde{d}$).  The other entries in the
triple, $\tilde{s} \in W^{\tilde{d}+g}_{\tilde{D}}$ and 
$\tilde{t} \in W^{\tilde{d}+g+1}_{\tilde{D}}$, describe
$\tilde{D}$ in the usual way; in particular,
$\Divisor \tilde{s} = D + \tilde{D}$.

We note here that the vector space $W^{3g+1}_D$ is always base point
free; hence this space always
determines $D$, regardless of whether $D$ is semi-typical.  In
any case, having $\tilde{s}$ and $\tilde{t}$ on hand allows us to
compute any $W^N_D$ that we might need, for example if some results
end up not being typical and we have to resort to the general-purpose
algorithms of~\cite{KKMAsymptotic}.  We also note that most of the
algorithms are valid using the base point free space $W^{3g}_D$, but
that this would lead to a longer algorithm for the intersection in the
addflip algorithm below.  We leave it to the reader to estimate the
probability of failure in these algorithms over a large finite field,
in the spirit of Remark~\ref{remark4.2}.


\begin{theorem}
\label{theorem4.3}
The following are algorithms for typical divisor classes in the above
setting.
\begin{enumerate}
\item
\textbf{Addflip of two divisor classes:}
Let the classes be represented by the divisors $D$ and $D'$.  Find the
intersection 
$W^{3g+1}_D \intersect W^{3g+1}_{D'}$, as in Steps (2)--(4) of
Proposition~\ref{proposition3.4} with $N=3g+1$.  This succeeds if and
only if $D$ and $D'$ are disjoint, and $D+D'$ is semi-typical.  We
thus obtain $s,t \in W^{3g+1}_{D+D'}$ with $\deg s = 3g$, $\deg t = 3g+1$,
and $\Divisor s = D + D' + D''$.  Second,  
compute $W^{3g+1}_{D''}$ by flipping, as in
Proposition~\ref{proposition3.6}, with $(N,D,A)$ there replaced by
$(3g+1, D+D', D'')$ here.  This step succeeds precisely when
$D+D'$ is typical, and implies that $D''$ is semi-typical.  We then
return the triple $(W^{3g+1}_{D''},s,t)$.  This whole procedure succeeds if
and only if $D,D'$ are disjoint, and $D+D'$ is typical.
\item
\textbf{Doubleflip of a divisor class:}
Given a divisor $D$ representing the class, take the pair 
$(\tilde{s}, \tilde{t})$, 
and apply Proposition~\ref{proposition3.12} with $N=3g+1$ to obtain
$s,t \in W^{3g+1}_{2D}$.  Since $\tilde{D}$ is typical, success occurs
if and only if $2D$ is 
semi-typical.  Now carry out a flip of $2D$, using
Proposition~\ref{proposition3.6} with $N=3g+1$, to produce the desired
answer  $W^{3g+1}_{D''}$, and to certify that $2D$ is typical.  If successful,
return as before $(W^{3g+1}_{D''},s,t)$.  This whole procedure succeeds if
and only if $2D$ is typical.

Remark: The algorithm as stated contains some redundancy in the form of
repeated computations.  Suppose that at the stage \emph{prior} to
starting the doubleflip, we had $\tilde{s}, \tilde{t}$ and were about
to carry out a flip to find $W^{3g+1}_D$ to obtain the full triple
describing $D$.  Then it would be desirable to have some rudimentary
lookahead to see whether $D$ will be used as an input for a
doubleflip.  If so, we can once and for all carry out the algorithm of
Proposition~\ref{proposition3.12}, instead of first ending the
previous computation with a simple flip using
Proposition~\ref{proposition3.6}.  That way, using
Proposition~\ref{proposition3.12}, we simultaneously obtain both the
space $W^{3g+1}_D$ and the elements $s,t \in W^{3g+1}_{2D}$.
\item
\textbf{Negation of a divisor class:}
Given the space $W^{3g+1}_D$, take the first two elements
$s\in W^{2g}_D, t\in W^{2g+1}_D$ of the echelon basis, with 
$\Divisor s = D+A$, and flip as in Proposition~\ref{proposition3.6} with
$N=3g+1$.  This produces the space $W^{3g+1}_A$, while certifying
that $D$ was typical to begin with (we already know that $A$ is
typical, because $\tilde{D}$ is).
Return as output the triple $(W^{3g+1}_A,s,t)$.  This procedure
succeeds if and only if $D$ is typical.
\end{enumerate}
\end{theorem}

\subsection{A nontraditional modification}
\label{subsection4.3}

As a last setting in which we can carry out generic algorithms, we describe
a change of perspective to the method of
Subsection~\ref{subsection4.2}.  Instead of representing the class
$[D - g\Pinf] \in \Pic^0(C)$ by the triple
$(W^{3g+1}_D,\tilde{s},\tilde{t})$,  we can view the pair
$(\tilde{s}, \tilde{t})$ as itself representing $D$, via the fact that
its ``opposite'' divisor $\tilde{D}$ is determined by
$I_{\tilde{D}} = \tilde{s}\Raff + \tilde{t}\Raff$.  Alternatively, we
can recover $D$ from the identity
$I_D = \{\ell \in \Raff \mid \ell \tilde{t} \in \tilde{s}\Raff\}$.
This allows us to carry out in essence the same algorithms as in
Subsection~\ref{subsection4.2}, except that now each step begins with
our flipping $(\tilde{s}, \tilde{t})$ using either
Proposition~\ref{proposition3.6} or Proposition~\ref{proposition3.12},
depending on whether we wish to carry out an addflip or a doubleflip.
Then we omit the final flip from the algorithms in the previous
subsection.  (The same technique works for negation in this model.)
Thus we have just shifted our perspective on where the algorithms
start and stop, so we do not think of $(\tilde{s}, \tilde{t})$ as
being extra baggage that we carry around to speed up some
computations, but rather as the actual result.  This approach
nonetheless comes with two disadvantages.  The first, minor,
disadvantage, is that a pair $(\tilde{s}, \tilde{t})$ no longer
represents a (semi-typical) divisor $D$ uniquely, since there are many
choices of $\tilde{D}$ in the same divisor class
with $\deg \tilde{D} = 2g$.  However, we can
always test equality between $(\tilde{s}, \tilde{t})$ and
$(\tilde{s}', \tilde{t}')$ by flipping both and seeing if they yield
the same space $W^{3g+1}_D = W^{3g+1}_{D'}$.

The second, more serious, disadvantage is that at the moment when
we compute a pair $(\tilde{s}, \tilde{t})$, we have not yet certified
that $\tilde{D}$ is typical; this certification happens only after we
flip using Proposition~\ref{proposition3.6} or
Proposition~\ref{proposition3.12}.  Thus in  
case one of those two algorithms fails, we have no guarantee that
$\tilde{s},\tilde{t}$ are an IGS for $\tilde{D}$, and so we may lose
information about what element of the Jacobian we are working with.  In
that case, we would need to backtrack one full step in the computations to
recover the information, and then use a slower general-purpose algorithm.

In conclusion, it is perhaps better in an implementation to stick to the
approach of Subsection~\ref{subsection4.2}, with some lookahead to
determine what to do with a particular pair $(\tilde{s},\tilde{t})$.  For
purposes of reasoning about the algorithm, however, the point of view in
this subsection may be useful.

\section*{Appendix: Speedup of the algorithms for $C_{3,4}$ curves}
\setcounter{equation}{0}
\setcounter{theorem}{0}
\renewcommand{\thesection}{A}

In this appendix, we describe a method to combine the computations of
Sections 8, 9, and~10 of~\cite{FASKKM} into a single more efficient
computation. The context here is that we currently know
$s,t \in W^{10}_{D+D'}$ (where $D$ might equal $D'$, and we know that
$D+D'$ is semi-typical).  In our previous algorithms, we did two 
flips to the pair $\{s,t\}$ to obtain first $D''$ and then $D'''$.
By a modification of Propositions \ref{proposition3.9}
and~\ref{proposition3.12}, we can combine 
these two flips into one computation.  We do not know whether
these techniques generalize to give a certifiably correct result for other
curves, even though they will work generically.  In the setting of
$C_{3,4}$ curves, however, it is easy to analyze when a divisor is
typical, and we can show that the results obtained are correct.

To start, let us change notation to write $(\tilde{D},D,A)$ in this
appendix, instead of $(D+D',D'',D''')$ from~\cite{FASKKM}.  Thus we
have elements $s \in W^9_{\tilde{D}}, t \in W^{10}_{\tilde{D}}$ of the
form

\begin{equation}
\label{equationA.1}
\begin{split}
s &= x^3 + s_1 y^2 + s_2 xy + s_3 x^2 + s_4 y + s_5 x + s_6,\\
t &= x^2 y + t_1 y^2 + t_2 xy + t_3 x^2 + t_4 y + t_5 x + t_6,\\
&\text{with } \Divisor s = \tilde{D} + D,
        \quad \Divisor t = \tilde{D} + E,\\
& \deg \tilde{D} = 6,
        \quad \deg D = 3,
        \text{ and } \deg E = 4.\\
\end{split}
\end{equation}

Our goal is to find the ``flip'' $A$ of~$D$.  Thus $A$ is a divisor
with $\deg A = 3$, and we wish to compute
$F \in W^6_{A+D}, G_0 \in W^7_{A+E}$ for which 
$F t + G_0 s = 0$.  We want to do so while certifying in the process
that $\{s,t\}$ is an IGS for $\tilde{D}$ (i.e., $D$ and $E$ are
disjoint) and that $\tilde{D}$ is typical, whence so 
is $A$.  (The reason for writing $G_0$ is that the 
final $G \in W^7_A$ will be a slight modification.)
As usual, we will compute with the apparently weaker system of equations
$Ft + G_0 s \equiv 0 \bmod W^8$.
This amounts to finding a linear combination of $t,xt,yt,x^2 t$ and
$s,xs,ys,x^2 s, xys$ that vanishes when viewed in the quotient space
$W^{16}/W^8$.  We represent elements of this quotient space as column
vectors with respect to the basis
$\{x^3, x^2 y, xy^2, y^3, x^3 y, x^2 y^2, xy^3, y^4\}$, analogously to
equation~(17) of~\cite{FASKKM}.  Adapting the entries of the matrix~$N$ in
Lemma~8.1 of that article, we have that our desired images of
$t,xt,yt,x^2 t, s,xs,ys,x^2 s, xys$ are the columns $C_1, \dots, C_9$ of
the matrix
\begin{equation}
\label{equationA.2}
\begin{pmatrix}
0 & t_3 & 0 & t_5 & 1 & s_3 & 0 & s_5+q_2 & 0 \\ 
1 & t_2 & t_3 & t_4+q_2+t_3p_2 & 0 & s_2+p_2 & s_3 & s_4+p_1+s_3p_2 & s_5+q_2 \\ 
0 & t_1 & t_2 & p_1 & 0 & s_1 & s_2 & 0 & s_4+p_1 \\ 
0 & 0 & t_1 & t_3 & 0 & 1 & s_1 & s_3 & 0 \\ 
0 & 1 & 0 & t_2 & 0 & 0 & 1 & s_2+p_2 & s_3 \\ 
0 & 0 & 1 & t_1+p_2 & 0 & 0 & 0 & s_1 & s_2+p_2 \\ 
0 & 0 & 0 & 0 & 0 & 0 & 0 & 1 & s_1 \\ 
0 & 0 & 0 & 1 & 0 & 0 & 0 & 0 & 1 \\ 
\end{pmatrix}
\end{equation}
Note that these columns are different from those in Section~9
of~\cite{FASKKM}, where we had 11~columns representing elements of
$W^{17}/W^9$.  The $p_i$ and $q_i$ are constants arising from the
equation of the curve $C$.  We emphasize that we do not compute the above
matrix directly, since this would involve the two products $t_3 p_2$ and
$s_3 p_2$, which we do not need separately, but can fold into other parts
of the computation.

As in our earlier article, we count the complexity of a computation in
terms of the number of multiplications $M$ and inversions $I$ it takes in
the field $\fieldK$.  We ignore additions and subtractions, as well as
multiplications and divisions by $2$ in $\fieldK$; recall that we assume
in~\cite{FASKKM} that $\fieldK$ does not have characteristic $2$ or~$3$.

The first stage of the computation is to compute three quantities $\ell_1,
\ell_2, \ell_3$ that will be useful later, for which nonvanishing of
$\ell_1$ is equivalent to $\tilde{D}$ (and hence $A$) being typical:
\begin{lemma}
\label{lemmaA.1}
Using $3M$, we can compute
\begin{equation}
\label{equationA.3}
\ell_1 = t_1 - s_2 + s_1^2,
\qquad
\ell_2 = t_2 - s_3 + s_1(s_2 + p_2),
\qquad
\ell_3 = t_3 + s_1 s_3.
\end{equation}
We then have:
\begin{enumerate}
\item
The combination of columns $C'_2 = C_2 - C_7 + s_1 C_6$, which represents 
$xt - ys + s_1 xs$, is equal to the column vector
$(\ell_3, \ell_2, \ell_1, 0, 0, 0, 0, 0)^\mathbf{T}$;
\item
Similarly, $C'_4 = C_4 - C_9 + s_1 C_8$, which represents 
$x(xt - ys + s_1 xs)$, has the form
$(*,*,*,\ell_3, \ell_2, \ell_1, 0, 0)^\mathbf{T}$;
\item
The divisor $\tilde{D}$ is typical if and only if $\ell_1 \neq 0$.
\end{enumerate}
\end{lemma}
\begin{proof}
Statements (1) and~(2) are direct computations.  Only statement~(3),
about typicality, needs proof.  Now $\tilde{D}$ is typical if and only
if we have invertibility of the $7\times 7$ submatrix 
of~\eqref{equationA.2} obtained from the columns corresponding to
$\{t, xt, yt, s, xs, ys, x^2s\}$ and the first seven rows, since this
corresponds to having $sW^6 + tW^5 + W^8 = W^{15}$.  The columns in
question are all except $C_4$ and $C_9$, and we can further replace
$C_2$ by $C'_2$, as given above, without affecting the invertibility; but
in that case the columns can be rearranged to form a triangular matrix with
diagonal entries all $1$, except for a single $\ell_1$.  This proves our
result. 
\end{proof}

We now define four more quantities $m_0, m_1, m_2, m_3$ by:
\begin{equation}
\label{equationA.4}
\begin{split}
m_0 &= \ell_3 - \ell_1 t_1,\\
m_1 &= -s_4 - (\ell_1 t_2 + \ell_2 t_1) - m_0 s_1,\\
m_2 &= t_4 - s_5 + s_1 (s_4 + p_1) + p_2 \ell_3 
            - (\ell_1 t_3 + \ell_2 t_2) - m_0 (s_2 + p_2), \\
m_3 &= t_5 + s_1(s_5 + q_2) - \ell_2 t_3 - m_0 s_3.\\
\end{split}
\end{equation}

The motivation for the above quantities is that 
$C''_4 = C'_4 - \ell_1 C_3 - \ell_2 C_2$ has the form
$(*,*,*,m_0,0,0,0,0)^\mathbf{T}$, while
$C'''_4 = C''_4 - m_0 C_6 = (m_3, m_2, m_1, 0,0,0,0,0)^\mathbf{T}$.
However, this fact is not needed to verify our proof below. 

\begin{lemma}
\label{lemmaA.2}
One can compute $m_0, \dots, m_3$ using only $10M$, as opposed to the $12M$
apparent in~\eqref{equationA.4}.
\end{lemma}
\begin{proof}
The point is that the four expressions 
$\alpha = \ell_1 t_1, \beta = \ell_1 t_2 + \ell_2 t_1,
 \gamma = \ell_1 t_3 + \ell_2 t_2, \delta = \ell_2 t_3$
can be computed using just $4M$ instead of the apparent $6M$.  This is
equivalent to Toom-Cook multiplication of polynomials via
interpolation at $0$, $1$, $-1$, and ``$\infty$''.  Explicitly, use
$4M$ to compute
$t_1 \ell_1$, $\ell_2 t_3$, $(t_1 + t_2 + t_3)(\ell_1 + \ell_2)$,
and  $(t_1 - t_2 + t_3)(\ell_1 - \ell_2)$.  Thus we know the quantities
$\alpha, \delta, \alpha+\beta+\gamma+\delta,
 \alpha - \beta + \gamma - \delta$.  Hence we also know $\beta \pm \gamma$
at no extra cost (of multiplications $M$), and can determine
$\beta, \gamma$ at no further 
cost, because division by $2$ is also ``free'' in our model.
\end{proof}

\begin{proposition}
\label{propositionA.3}
Given $\ell_1, \ell_2, \ell_3, m_0, m_1, m_2, m_3$ as above, one can at a
further cost of $1I, 4M$ compute $\ell_1^{-1}, m_1/\ell_1,
(m_1/\ell_1)\ell_2, (m_1/\ell_1)\ell_3, (m_1/\ell_1)s_1$, thereby obtaining
the following values of $F, G_0$:
\begin{equation}
\label{equationA.5}
\begin{split}
F &= x^2 - \ell_1 y - (\frac{m_1}{\ell_1} + \ell_2) x
        + (\frac{m_1}{\ell_1}) \ell_2 - m_2,\\
G_0 &= -xy + s_1 x^2 + (\frac{m_1}{\ell_1})y
           - (m_0 + (\frac{m_1}{\ell_1})s_1) x 
           + (\frac{m_1}{\ell_1}) \ell_3 - m_3.\\
\end{split}
\end{equation}
Being able to invert $\ell_1$ certifies that $\tilde{D}$ is
typical, and that the above computation correctly finds $F \in W^6_{D}$.
Writing $\Divisor F = D+A$, we also obtain that $\Divisor G_0 = A+E$, and
that the pair $(F,G)$ with $G = -G_0 + s_1 F$ is an IGS for the typical
divisor $A$.  It costs a further $3M$ to compute the coefficients of
$G$ from $F$ and $G_0$.  Thus the total cost of this proposition is
$1I, 7M$, if done in two stages.  However, it is possible to bring the
total cost down to $1I, 6M$, by combining both parts of the
computation to yield $F$ and $G$ directly.
\end{proposition}
\begin{proof}
One can check by a lengthy calculation (preferably using a computer)
that $F t + G_0 s \equiv 0 \bmod W^8$; this amounts to checking that
the appropriate linear combination of columns of~\eqref{equationA.2}
vanishes. We have already shown that invertibility of $\ell_1$ implies
that $\tilde{D}$ is typical.  This implies that the divisors $D$ and $E$ 
from~\eqref{equationA.1} are disjoint, and that $W^8_{\tilde{D}} = 0$,
so we obtain as usual that $F t + G_0 s = 0$, and that $F \in W^6_D$.
The statement about $\Divisor G_0$ follows.  Computing $G$ from $G_0$
involves $3M$ because we need to multiply $s_1$ by each of the coefficients 
$\ell_1, (\frac{m_1}{\ell_1} + \ell_2),
((\frac{m_1}{\ell_1}) \ell_2 - m_2)$ of $F$.  We thus obtain a pair $(F,G)$
in $W^7_A$ whose $F$ has a coefficient $-\ell_1$ for the $y$ monomial.
Thus we have obtained a description of the divisor $A$ as in~\cite{FASKKM},
with the equivalent of $a\neq 0$ from Proposition~\ref{proposition2.9}, and
no added cost to compute $a^{-1} = -\ell_1^{-1}$.

We now explain the extra saving of $1M$ from folding the computations
together.  This comes from the coefficient of $x$ in $G$.  As stated
currently, it appears to take $2M$ to compute this coefficient: (i) the
first $M$ comes from the multiplication $(m_1/\ell_1)\cdot s_1$, to
compute the coefficient of $x$ in $G_0$, which is 
$-(m_0 + (m_1/\ell_1)s_1)$; (ii) the second $M$ comes when we compute
$G = -G_0 + s_1 F$, since we multiply $s_1$ by the coefficient of $x$
in $F$, which is $-((m_1/\ell_1) + \ell_2)$.  However it is immediate
that the coefficient of $x$ in $G$ that results from this is
\begin{equation}
\label{equationA.6}
-(-(m_0 + (m_1/\ell_1)s_1)) + s_1(-((m_1/\ell_1) + \ell_2))
 = m_0 - s_1 \ell_2,
\end{equation}
which can naturally be computed using the single $M$ of $s_1 \cdot
\ell_2$.  This concludes the proof.
\end{proof}

Combining Lemmas \ref{lemmaA.1} and~\ref{lemmaA.2} with
Proposition~\ref{propositionA.3}, we obtain the following result:

\begin{theorem}
\label{theoremA.4}
The above procedure produces the same effect as Proposition~9.3 and
Proposition~10.1(i) of~\cite{FASKKM}.  This means that we can use a total
of $19M, 1I$ to replace what took us $38M, 1I$ in~\cite{FASKKM}.
Consequently, the cost of Jacobian operations in a $C_{3,4}$ curve can be
reduced by $19M$ to obtain that addition of typical elements can be carried
out using $98M, 2I$ while doubling can be carried out using $110M, 2I$.
The results are certified to be correct and typical, provided all
inverses can be computed.  This represents a further speedup of
approximately 15\% over the results of that article.
\end{theorem}


\bibliographystyle{amsalpha}

\bibliography{article}

\end{document}